\documentclass[10pt]{amsart}

\usepackage{amssymb, amsmath, amsthm}
\newtheorem{thm}{Theorem}[section]
\newtheorem{lem}[thm]{Lemma}

\newtheorem{prop}[thm]{Proposition}
\newtheorem{defn}[thm]{Definition}
\newtheorem{rmk}{Remark}

\numberwithin{equation}{section}

\newcommand{\bel}{\begin{equation} \label}
\newcommand{\ee}{\end{equation}}
\def\beq{\begin{equation}}
\def\eeq{\end{equation}}
\newcommand{\bea}{\begin{eqnarray}}
\newcommand{\eea}{\end{eqnarray}}
\newcommand{\beas}{\begin{eqnarray*}}
\newcommand{\eeas}{\end{eqnarray*}}
\newcommand{\pd}{\partial}

\newcommand{\R}{\mathbb{R}}

\newcommand{\re}{\mathfrak R}
\newcommand{\im}{\mathfrak I}

\newcommand{\supp}{\mathrm{supp}\,}  

\allowdisplaybreaks

\def\epsilon{\varepsilon}
\def\phi {\varphi}

\providecommand{\abs}[1]{\left\lvert#1\right\rvert}
\providecommand{\norm}[1]{\left\lVert#1\right\rVert}

\renewcommand{\leq}{\leqslant}
\renewcommand{\geq}{\geqslant}
\providecommand{\abs}[1]{\left\lvert#1\right\rvert}
\providecommand{\norm}[1]{\left\lVert#1\right\rVert}

\usepackage{geometry}
\newgeometry{asymmetric}

\usepackage[colorlinks=true, citecolor=blue]{hyperref}

\def\p{\pd}

\def\M{\overline M}
\def\g{{\overline g}}

\title{Recovery of time-dependent coefficient on Riemanian manifold for hyperbolic equations}
\author{Yavar Kian, Lauri Oksanen}

\date{}
\begin{document}
\maketitle
\begin{abstract}
Given $(M,g)$, a compact connected Riemannian manifold of dimension $d \geq 2$, with boundary $\partial M$, we study 
the inverse boundary value problem of  determining   a time-dependent potential $q$, appearing in   the wave equation $\partial_t^2u-\Delta_g u+q(t,x)u=0$ in $\M=(0,T)\times M$ with $T>0$.
Under suitable geometric assumptions we prove global unique determination of  $q\in L^\infty(\M)$ 
given the Cauchy data set on the whole boundary $\p \M$,
or on certain subsets of $\p \M$.
Our problem can be seen as an analogue of the Calder\'on problem on the Lorentzian manifold $(\M, dt^2 - g)$. 


\end{abstract}

\section{Introduction}
\subsection{Formulation of the problem}
Let $(M,g)$ be a smooth Riemanian manifold with boundary and let $T>0$.  We introduce the Laplace and wave operators
\begin{align}
\label{def_Laplace}
\Delta_{g} u = |g|^{-1/2} \p_{x^j}
 \left( g^{jk} |g|^{1/2} \p_{x^k} u\right),
 \quad 
 \Box_g=\partial_t^2-\Delta_g,
\end{align}
where $|g|$ and $g^{jk}$ denote the absolute of value of the determinant and the inverse of $g$
in local coordinates, and consider the wave equation 
\begin{equation}\label{wave}\Box_g u+q(t,x)u=0,\quad (t,x)\in (0,T)\times M,\end{equation}
with $q\in L^\infty((0,T)\times M)$. Let $\nu$ be the outward unit normal vector to $\partial M$ with respect to the metric $g$ and let $\partial_\nu$ be the corresponding normal derivative. 
Moreover, we define $\p_{\overline \nu} = \p_\nu$ on the lateral surface 
$(0,T) \times \p M$, $\p_{\overline \nu} = \p_t$ on the top surface $\{T\} \times M$ and $\p_{\overline \nu} = -\p_t$ on the bottom surface $\{0\} \times M$, and consider the Cauchy data set on the boundary of the cylinder $\M = (0,T) \times M$,
\bel{cq}\mathcal C_{q}=
\{(u_{|\p \M},\partial_{\overline \nu}u_{|\p \M}):\  u\in L^2(\M),\ \Box_gu+qu=0\}.\ee
In this paper we study the inverse boundary value problem to recover the time-dependent zeroth order term $q$ appearing in \eqref{wave} from partial knowledge of the set $\mathcal C_{q}$.

There are several previous results on the problem, and we shall review them below, however to our knowledge, all of them assume either that $(M,g)$ is a domain in $\R^n$ with the Euclidean metric or that $q$ is time-independent. 

In the case of time-independent $q$ it is enough to know the 
following lateral restriction of $\mathcal C_{q}$,
\begin{equation*}
\mathcal C_q^{\text{Lat}}=\{(u_{|(0,T)\times \partial M},\partial_\nu u_{|(0,T)\times \partial M}):\ u\in L^2(\M),\ \Box_g u+qu=0,\ u_{|t=0}=\partial_tu_{|t=0}=0\},
\end{equation*}
for sufficiently large $T > 0$, in order to determine $q(x)$ for all $x \in M$, see \cite{[BD], Katchalov2001,Mo,RS1}.
However, if $q$ depends on time, due to domain of dependence argument,
the data $\mathcal C_q^{\text{Lat}}$ contains no information on
the restriction of $q$ on the set 
\def\dist{\textrm{dist}}
\begin{align}
\label{lateral_obstruction}
\{(t,x)\in \M:\ 
\dist(x,\partial M)> t \text{ or } \dist(x,\partial M) > T - t\}.
\end{align} 
Here $\dist(\cdot, \cdot)$ is the distance function on $(M,g)$.
Facing this obstruction to the uniqueness, 
all the results in the present paper assume some information 
on the top $\{T\} \times M$ and bottom $\{0\} \times M$ surfaces. 
In particular, 
under the assumption that $(M,g)$ is a simple manifold, see Definition \ref{def_simple} below, 
we show that the full Cauchy data set $\mathcal C_{q}$
determines $q$ uniquely. The precise formulations of our results, with partial knowledge of the set $\mathcal C_{q}$, are in Section \ref{sec_main_results} below.

\subsection{Physical and mathematical motivations}
\label{sec_motivation}

Let us begin with a mathematical motivation:
the problem to determine $q$ given $\mathcal C_{q}$
can be seen as a hyperbolic analogy of the Calder\'on problem on a cylinder as stated in \cite{DKSU}.
Indeed, denoting by $dt^2 - g$ the product Lorentzian metric on $\M$, the wave operator $\Box_g$ coincides with the 
Laplace operator on $(\M, dt^2 - g)$.
On the other hand, denoting by $\g = dt^2 + g$
the Riemannian product metric on $\M$, and choosing a smooth domain 
$\Omega \subset \M$, we can formulate the the Calder\'on problem on a cylinder as follows: given the elliptic Cauchy data set 
$$
\mathcal C_{q}^{\text{Ell}} = 
\{(u_{|\p \Omega},\partial_{\overline \nu}u_{|\p \Omega}):\  u\in L^2(\Omega),\ \Delta_\g u+qu=0\}
$$
determine $q$ (here $\overline \nu$ is the outward unit normal vector to $\p \Omega$). In \cite{DKSU} this problem was solved under the assumption that $(M,g)$ is a simple manifold. 

One reason to study these problems 
is to gain some understanding of the fundamental problem to determine, up to an isometry, a smooth Riemannian or Lorentzian manifold $(\Omega, \g)$ with boundary given the set of Cauchy data
$$
\mathcal C(\g) = \{(u|_{\p \Omega}, \p_{\overline \nu} u|_{\p \Omega});\ u \in L^2(\Omega),\ \Delta_\g u =0 \}.
$$
Excluding results where full or partial real analyticity is assumed, this problem is open in dimensions three or higher, in both the elliptic and hyperbolic cases. 
The relation to the present problem to determine $q$ given $\mathcal C_{q}$ is as follows.
In the case when $(\Omega, \g)$ is a subset of the conformal cylinder 
\begin{align}
\label{cylinder}
\M = (0,T) \times M, \quad
\g = c(dt^2 + g),
\end{align}
where only the strictly positive conformal factor $c \in C^2(\M)$ is assumed to be unknown, the problem to determine $c$ given $\mathcal C(\g)$ can be reduced to the problem to determine $q$ given $\mathcal C_{q}^{\text{Ell}}$ via a gauge transformation. 
Indeed, as explained e.g. in \cite{DKLS},
the function $v = c^{(d-1)/4} u$
satisfies $\Delta_{\g} v + q_c v = 0$
if the function $u$ satisfies $\Delta_{\g} u = 0$, where $d$ is the dimension of $M$
and 
$$
q_c = c^{-(d-1)/4} \Delta_{\g} c^{(d-1)/4}.
$$ 
This allows 
us to first determine $\mathcal C_{q_c}^{\text{Ell}}$ given $\mathcal C(\g)$, then
solve the inverse boundary value problem for $q_c$, and finally determine $c$ given $q_c$. The argument can be adapted also to the hyperbolic case.

From the physical point of view, our inverse problem consists of determining properties such as the time evolving  density of an inhomogeneous medium by probing it with disturbances generated on  some parts of the boundary and at  initial time, and by measuring  the response  on some parts of the boundary and at the end of the experiment. 


Time-dependent zeroth order terms appear often also due to mathematical reductions of nonlinear problems. 
For example, in \cite{I2} Isakov applied results on inverse boundary value problems with time-dependent coefficients in order to prove  unique recovery of a general semilinear term appearing in
a nonlinear parabolic equation from traces of all the solutions to the equation. More recently, applying their results of stable recovery of time-dependent coefficients from the parabolic Dirichlet-to-Neumann map, \cite{CK2} treated  the stability issue for this problem.
In the same spirit our inverse problem  can  be a tool  for the  problem of determining a semilinear term appearing in a nonlinear wave equation from observations given by traces of the solutions. We point out that
with this application in mind, it is important 
to treat non-smooth potentials $q$.

\subsection{Previous literature}

The recovery of coefficients appearing in hyperbolic equations is 
a topic that has attracted considerable attention. 
Several authors have treated the determination of time-independent coefficients from Cauchy data analogous to $\mathcal C_q^{\text{Lat}}$ above. 
In this case, the Boundary Control method, originating from \cite{Bel87}, 
gives very general uniqueness results when combined with 
the time-sharp unique continuation theorem \cite{T}.
We refer to \cite{Kurylev2015,LO} for state-of-the-art results and to
\cite{Belishev2007,Katchalov2001}
for reviews. However, as shown in \cite{Alinhac1983}, unique continuation analogous to \cite{T} may fail in the presence of time-dependent zeroth order terms, and the Boundary Control method generalizes only to the case where the dependence on time is real analytic 
\cite{E2}.

Another approach, that is constrained in the time-independent case, is
the Carleman estimates based approach originating from \cite{BK}.
Using this approach it can be shown, under suitable geometric assumptions, that a single well-chosen element $\mathcal C_q$ determines $q$.
The approach gives also strong stability results.

Let us now turn to the approach underpinning most of the results in the time-dependent case, including the results in the present paper, that is, the use of geometric optics solutions. 
This approach is widely applied also to time-independent case,
and the data used then is typically the same as in the case of the Boundary Control method, that is, $\mathcal C_q^{\text{Lat}}$.
Although the geometric optics approach gives less sharp uniqueness results in terms of geometrical assumptions than the Boundary Control method, the advantage of the former is that it yields stronger stability results \cite{[BD],BJY,Ki1,Mo,SU,SU2} than the latter \cite{Anderson2004},
however, see \cite{LO2} for a strong low-frequency stability result using ideas from the Boundary Control method. 

Apart from \cite{E2}, all the above results are concerned with time-independent coefficients.
The geometric optics approach in the time-dependent case was
first used by Stefanov. In \cite{St} he determined a time-dependent potential in a wave equation, with constant coefficients in the leading order, from scattering data 
via a reduction to the light-ray transform in the Minkowski space. 
A similar strategy was used by Ramm and Sj\"ostrand \cite{RS},
the difference being that instead of scattering data they used the analogy of the lateral data $\mathcal C_q^{\text{Lat}}$ in the 
infinite cylinder $\R \times M$ where $M$ is a domain in $\R^n$.
Rakesh and  Ramm \cite{RR} considered the lateral data $\mathcal C_q^{\text{Lat}}$ on a finite cylinder $(0,T)\times M$ and they determined $q$ in a subset of $(0,T)\times M$ contained in the complement of (\ref{lateral_obstruction}).
In \cite[Theorem 4.2]{I}, Isakov 
determined $q$ on the whole domain $(0,T)\times M$ from the full data $\mathcal C_{q}$.
In \cite{Ki2,Ki3}, the first author established both uniqueness and stability for the determination of a general time dependent potential $q$ from (roughly speaking) half of the data in \cite{I}.  More recently, \cite{Ki4} extended the result of \cite{Ki2,Ki3} to recovery of time-dependent damping coefficients,
and Salazar \cite{S} generalized the result of \cite{RS} to recovery of time-dependent magnetic vector potentials. 
We mention also the log-type stability result \cite{A} for the determination of time-dependent potentials from the data considered by \cite{I} and \cite{RR}.

The above results \cite{A,I,Ki2,Ki3,Ki4,RR, RS,S, St} 
all assume that  the leading order coefficients in the wave equation are constant. 
The main contribution of the present paper is to consider the recovery of a time-dependent potential in the case of non-constant leading order coefficients. 


\subsection{Main results}
\label{sec_main_results}

We prove two results on unique determination of the potential $q$. In the first result we assume that the Cauchy data set $\mathcal C_q$ is fully known on the lateral boundary $(0,T) \times \p M$ and partly restricted on the top and bottom. 
In the second result we restrict the data also on the lateral boundary. 
In both the results we impose geometric conditions on the manifold $(M,g)$, the conditions being more stringent in the second case. 
In the first case, we make the typical assumption that $(M,g)$ is simple in the sense of the following definition.

\begin{defn}
\label{def_simple}
A compact smooth Riemannian manifold with boundary $(M,g)$ is simple if it is simply connected, the boundary $\p M$ is strictly convex in the sense of the second fundamental form, and $M$ has no conjugate points. 
\end{defn}

We consider the restricted version of $\mathcal C_q$,
$$
\mathcal C(q, 0) = \{(u|_{\p \M \setminus(\{0\}\times M)}, \p_{\overline \nu} u|_{\p \M});
\  u\in L^2(\M),\ \Box_gu+qu=0,
\ u_{|t=0}=0\},
$$
and formulate our first result.
\begin{thm}\label{t1} 
Suppose that $(M,g)$ is a simple. Let $T > 0$ and 
let $q_1$, $q_2\in  L^\infty((0,T)\times M)$. 
Then
\bel{t1a}\mathcal C(q_1, 0)=\mathcal C(q_2, 0)\ee
implies that $q_1 = q_2$.
\end{thm}

Let us point out that an analogous result holds with the data restricted on the top $\{T\} \times  M$ rather than on the bottom $\{0\} \times M$,
and also with the time derivative $\p_t u|_{t = T}$ vanishing instead of $u|_{t=0}$. Moreover, we prove also a variation of Theorem \ref{t1} 
using the data 
$$
\mathcal C(q, 0,T) = \{(u|_{\p \M\setminus(\{0\}\times M)}, \p_{\overline \nu} u|_{\p \M\setminus(\{T\}\times M)});
\  u\in L^2(\M),\ \Box_gu+qu=0,
\ u_{|t=0}=0\}.
$$
\begin{thm}\label{t1var} 
Let $(M,g)$ be simple. Let $T>\textrm{Diam}(M)$ and 
let $q_1$, $q_2\in  L^\infty((0,T)\times M)$. 
Then
\bel{Tt1d}\mathcal C(q_1, 0,T)= \mathcal C(q_2, 0,T)\ee
implies that $q_1 = q_2$.
\end{thm}

In order to restrict the data also on the lateral part of the boundary, we make the  assumption that $(M,g)$ is contained in a conformal cylinder of the form (\ref{cylinder}), that is, we assume that it satisfies the geometric assumption introduced in \cite{DKSU} in the context the Calder\'on problem. Furthermore, we assume that 
also the time direction is multiplied by the same conformal factor, 
which amounts to assuming, after the gauge transformation discussed in Section \ref{sec_motivation}, that the wave equation has two Euclidean directions, one of them being the time direction. 

More precisely, we assume that 
$(M_0,g')$ is a simple Riemanian manifold of dimension $d-1\geq2$,
$M \subset \R \times int(M_0)$ is a compact domain with smooth boundary,
and that $g=a(e\oplus g')$ where $e$ is the euclidean metric on $\R$
and $a \in C^\infty(M)$ is positive,
and consider the wave operator
\begin{align}
\label{wave_eucl2}
\Box_{a,g}=a^{-1}\pd_t^2-\Delta_g.
\end{align}
Let us now describe the restriction of $\mathcal C_q$
considered in our second result. 
To every variable $x\in M$ we associate the coordinate $x_1\in\R$ and $x'\in M_{x_1}=\{x'\in M_0:\ (x_1,x')\in M\}$ such that $x=(x_1,x')$.
We define 
$\phi(x)=x_1$, 
$$
\pd M_\pm=\{x\in\pd M:\ \pm \pd_\nu\phi(x)\geq0\},
$$ 
and $\Sigma_\pm=(0,T)\times int(\pd M_\pm)$.
We consider $U=[0,T]\times U'$ (resp $V=(0,T)\times V'$) with $U'$ (resp $V'$) a closed  neighborhood of $\partial M_{+}$ (resp $\partial M_-$) in $\partial M$, and define the following restriction of $\mathcal C_q$,
\[\mathcal C_{q,*}=\{(u_{|U},\partial_t u_{|t=0},\partial_\nu u_{|V}, u_{|t=T}):\ u\in L^2(\M),\ (\Box_{a,g} +q)u=0,\ u_{|t=0}=0,\ \textrm{supp}u_{|(0,T)\times\pd M}\subset U\}.\]
Our second result is stated as follows.

\begin{thm}\label{thm1} 
Suppose that the leading part of the wave operator is of the form (\ref{wave_eucl2}). Let $T > 0$ and let 
$q_1,\ q_2 \in L^\infty((0,T)\times M)$. Then
\begin{equation}
\label{thm1a} 
\mathcal C_{q_1,*}=\mathcal C_{q_2,*}
\end{equation}
implies that $q_1=q_2$.
\end{thm}

\subsection{Remarks about the proofs of the main results}

As indicated above, the proofs of Theorems \ref{t1} and \ref{thm1}
are based on the use of geometric optics solutions. In the case of the former, we use the oscillating solutions of the form 
\bel{osc}u(t,x)=\sum_{j=1}^ka_j(t,x)e^{i\sigma \psi_j(t,x)}+R_\sigma(t,x),\quad (t,x)\in(0,T)\times M,\ee
with $\sigma\in\R$ a  parameter, $R_\sigma$ a term that admits a decay with respect to the parameter $|\sigma|$ and $\psi_j$, $j=1,..,k$,  real valued. Inspired by the elliptic result \cite{DKSU}, we use these solutions to prove that the hyperbolic inverse boundary value problem reduces to the problem to invert a weighted geodesic ray transform on $(M,g)$.
The assumption that $(M,g)$ is simple guarantees that this transform is indeed invertible.

For our purposes it is enough to take $k=2$ in (\ref{osc}), and in the case of full data $\mathcal C_q$ already $k=1$ is enough.
In the case of data sets $\mathcal C(q, 0)$ and $\mathcal C(q, 0,T)$,
the second term is needed in order to be able to restrict the data while  avoiding a "reflection". 
Similar construction is likely to work also on the lateral boundary,
and one may hope that this could be used to reduce the amount of lateral data. In fact, this type of argument was used in the elliptic case in \cite{KS}. There it was assumed that the part of the lateral boundary lacking data, that is, the inaccessible part, 
satisfies a (conformal) flatness condition in one direction,
and the elliptic inverse boundary value problem was reduced to the invertibility of a broken geodesic ray transform. The geodesics used in the transform break via the normal reflection when they hit the inaccessible part of the boundary. 
However, barring
some special cases, it is not known if such a transform is invertible,
and moreover, there are also counter-examples to invertibility in
general. We refer to \cite{Il} for a discussion of both positive results
and counter-examples, and do not pursue a lateral reflection type argument in the present paper. 

We recall that several authors, e.g. \cite{A,RR, RS,S}, have 
treated the problem to determine a time-dependent potential in a wave equation. In these results $(M,g)$ is a domain with the Euclidean geometry, and the proofs are based on the use of solutions of the form (\ref{osc}) to reduce the inverse boundary value problem to the problem to invert the light-ray transform in the Minkowski space. An analogous reduction is possible also in the case of more complicated geometry \cite{W}, however, 
it is an open question if the light-ray transform is invertible on a Lorentzian manifold of the product form $((0,T) \times M, dt^2 - g)$
where $(M,g)$ is simple.  
We remark that in the case of a real-analytic Lorentzian manifold satisfying a certain convexity condition, the invertibility is shown in the recent preprint \cite{S_lr}. 
We do not pursue this direction in the present paper, as having (restricted) data on the top and bottom allows for a reduction to the well-understood problem to invert a weighted geodesic ray transform,  
rather than the light-ray transform.

For Theorem \ref{thm1}, inspired by  \cite{BJY,Ki2,Ki3,Ki4}, we replace the oscillating solutions (\ref{osc})
by exponentially growing or decaying solutions of the form
\bel{exp}u(t,x)=e^{\sigma (\beta t+\phi(x))}(a_\sigma(t,x)+R_\sigma(t,x)),\quad (t,x)\in(0,T)\times M,\ee
with $\sigma\in\R$ a  parameter, $\beta\in(0,1]$, $R_\sigma$ a term  that  admits a decay with respect to the parameter $|\sigma|$ and $ \phi$ a limiting Carleman weight for elliptic equations as defined in \cite{DKSU}. 
The geometric assumption in Theorem \ref{thm1}
guarantees that the "phase" function $\beta t+\phi(x)$ 
is characteristic in the sense that its gradient is in the characteristic set of the wave operator, and this allows us to construct solutions of the form (\ref{exp}). The construction is based on a Carleman estimate with the above characteristic phase as the Carleman weight. 
The results \cite{BU,DKSU,KSU} can be viewed as elliptic analogies of such a construction.

Let us also emphasize that   we propose a construction of  geometric optics solutions taking into account both the low regularity of the time-dependent potential and the geometrical constraints.

\subsection{Outline} 
This paper is organized as follows. Section 2 is devoted to the proof Theorem \ref{t1} whereas Sections 3, 4, 5 and 6 are concerned with Theorem \ref{thm1}. 
In Section 3 we give some preliminary results on solving the direct problem in $L^2(\M)$, that is, with smoothness below the natural energy level. 
This is needed for certain duality arguments in the proof of Theorem \ref{thm1}.
In Section 4 we prove a Carleman estimate with the characteristic weight function, and in Section 5 we build exponentially growing and decaying geometric optics solutions designed in accordance with the estimate. Finally combining the results of Sections 4 and 5, we complete the proof of Theorem \ref{thm1} in Section 6.

\section{The problem with full data on the lateral boundary}

We begin by writing the Cauchy data set $\mathcal C_q$ as a graph.
For the purposes of Theorem \ref{t1} it is enough to consider energy class solutions, that is, solutions in 
$$
\mathcal H = C^1([0,T]; L^2(M))\cap C([0,T];H^1(M)).
$$
From the point of view of uniqueness results such as Theorems \ref{t1} and \ref{thm1},
the restriction of $\mathcal C_q$ to the energy class, that is,
$$
\widetilde {\mathcal C_q} = 
\{(u_{|\p \M},\partial_{\overline \nu}u_{|\p \M}):\  u\in \mathcal H,\ \Box_gu+qu=0\},
$$
makes no difference, since it can be shown that $\widetilde {\mathcal C_q}$
is dense in $\mathcal C_q$, say, in the sense of distributions, and whence $\widetilde {\mathcal C_q}$ determines $\mathcal C_q$.

For $T>0$ and $q\in L^\infty((0,T)\times M)$ we consider the initial boundary value problem 
\bel{eeq1}
\left\{ \begin{array}{ll}  \pd_t^2 u-\Delta_{g} u + q u  =  0, & \mbox{in}\ (0,T)\times M,\\  
u  =  f, & \mbox{on}\ (0,T)\times \partial M,\\  
 u(0,\cdot)  =v_0,\quad  \pd_tu(0,\cdot)  =v_1 & \mbox{in}\ M,
\end{array} \right.
\ee
with non-homogeneous Dirichlet data $f$ and initial conditions $v_0,v_1$,
and define the boundary operator
\[B_{q}:(f,v_0,v_1)\mapsto (\partial_\nu u_{|(0,T)\times\partial M},u_{|t=T},\pd_tu_{|t=T}),\]
where $u$ solves problem \eqref{eeq1}.
It follows from \cite{LLT} that 
$B_q$ is continuous from the space of functions 
$$
(f,v_0,v_1) \in H^1((0,T)\times\pd M) \times H^1(M) \times L^2(M),
$$
satisfying the compatibility condition $f=v_0$ in $\{0\} \times \p M$,
to the space 
$$L^2((0,T)\times\pd M) \times H^1(M) \times L^2(M).$$
The Cauchy data set $\widetilde{\mathcal C_q}$ is the graph of $B_q$.

In order to highlight the main ideas of the proof of Theorem \ref{t1},
we consider first the full data case, that is, we begin by proving the following theorem: 
\begin{thm}\label{t1full} 
Suppose that $(M,g)$ is a simple. Let $T > 0$ and 
let $q_1$, $q_2\in  L^\infty((0,T)\times M)$. 
Then $B_{q_1} = B_{q_2}$
implies that $q_1 = q_2$.
\end{thm}

\subsection{Geometric optics solutions }

The goal of this subsection is to construct energy class solutions to the wave equation of the form
\
\bel{gGO1} u(t,x)= a(t,x)e^{i\sigma(\psi(x)+ t)}+R_\sigma(t,x), \quad |\sigma| > 1,
\ee
with the remainder term $R_\sigma\in \mathcal H$ satisfying
\begin{align}
\label{gGO}
&R_\sigma=0 \textrm{ on }(0,T)\times\pd M,\quad R_\sigma(s,\cdot)=\pd_tR_\sigma(s,\cdot)=0 \textrm{ on } M,
\\\label{gGO2}
&\lim_{|\sigma|\to+\infty} \norm{R_\sigma}_{L^2((0,T)\times M)}=0,
\end{align}
where either $s=0$ or $s=T$.

In order to get the decay \eqref{gGO2}, we choose $\psi$ and $a$ so that they satisfy the following eikonal and transport equations
\bel{psi}\sum_{i,j=1}^dg^{ij}(x)\pd_{x_i}\psi\pd_{x_j}\psi=|\nabla_g\psi|^2_g=1,\ee
\bel{a}  2i\pd_ta-2i\sum_{i,j=1}^dg^{ij}(x)\pd_{x_i}\psi\pd_{x_j}a-i(\Delta_g \psi)a=0.\ee
As $(M,g)$ is assumed to be simple, the eikonal equation can be solved globally on $M$.
To see this, first extend the simple manifold $(M,g)$ into a simple manifold $(M_1,g)$ in such a way that $M$ is contained into the interior of $M_1$. 
Now pick $y\in \pd M_1$ and consider the polar normal coordinates $(r,\theta)$ on $M_1$ given by $x=\exp_y(r\theta)$ where $r>0$ and $\theta\in S_y(M_1):=\{v\in T_yM_1:\ |v|_g=1\}$.  According to the Gauss lemma (see e.g. \cite[Chaper 9, Lemma 15]{Sp}), in these coordinates the metric takes the form $g(r,\theta)=dr^2+g_0(r,\theta)$ with $g_0(r,\theta)$ a metric on $S_yM_1$ that depends smoothly on $r$. 
We choose 
\begin{align}
\label{psi1}
\psi(x)=\textrm{dist}(y,x), \quad x \in M,
\end{align}
with $\textrm{dist}$ the Riemanian distance function on $(M_1,g)$. 
As $\psi$ is given by $r$ in the polar normal coordinates,  one can easily check that $\psi$ solves \eqref{psi}. 

\
Let us now turn to the transport equation.
We write $a(t,r,\theta)=a(t,\exp_y(r\theta))$
and use this notation to indicate the representation in the polar normal coordinates also for other functions.
Moreover, we define $b(r,\theta)=\textrm{det}g_0(r,\theta)$,
and transform \eqref{a} into

\[ 
\pd_ta-\pd_ra-\left({\pd_rb \over 4b}\right)a=0.\]

We see that for any $h\in C^\infty( S_yM_1)$, $\chi\in C^\infty(\R)$ and $\mu>0$
the function 
\bel{a1} a(t,r,\theta)=e^{-{\mu (r+t)\over2}}\chi(r+t)h(\theta)b(r,\theta)^{-1/4}\ee 
is a solution of the transport equation.

We are now ready for the construction of the remainder term. 
\begin{lem}\label{Ll1} Let $q\in  L^\infty((0,T)\times M)$.
Choose $\psi$ and $a$ by \eqref{psi1} and \eqref{a1} respectively.
Then there exists a solution $u \in \mathcal H$ of $\pd_t^2u-\Delta_gu+qu=0$ of the form \eqref{gGO1}
 where the remainder term $R_\sigma$ satisfies \eqref{gGO}-\eqref{gGO2}
and $\pd_\nu R_\sigma \in L^2((0,T) \times \pd M)$. 
 \end{lem}
\begin{proof}
Without loss of generality we assume that $s=0$ in \eqref{gGO} and that $\sigma > 0$.
By (\ref{psi}) and (\ref{a}),

\[(\partial_t^2-\Delta_g+q)\left(a(t,x)e^{i\sigma(\psi(x)+ t)}\right)={e^{i\sigma(\psi(x) + t)}(\partial_t^2-\Delta_g+q)a}.\]
We define $F_\sigma = e^{i\sigma \psi(x)}(\partial_t^2-\Delta_g+q)a$ and see that the remainder term must satisfy 
\bel{eq2_R}
\left\{ \begin{array}{ll}  \pd_t^2 R-\Delta_{g} R + q R  =  -e^{i\sigma t}F_\sigma, & \mbox{in}\ (0,T)\times M,\\  
R  =  0, & \mbox{on}\ (0,T)\times \partial M,\\  
 R(0,\cdot)  =0,\quad  \pd_tR(0,\cdot)  =0 & \mbox{in}\ M.
\end{array} \right.
\ee
As $F_\sigma\in L^2((0,T)\times M)$ we deduce from \cite[Theorem 2.1]{LLT} that (\ref{eq2_R}) 
admits a unique solution $R_\sigma\in \mathcal H$ satisfying $\pd_\nu R_\sigma\in L^2((0,T)\times \pd M)$ and the energy estimate
$$\norm{R_\sigma}_{\mathcal H}
+\norm{\pd_\nu R_\sigma}_{L^2((0,T)\times \pd M)} \leq C\norm{F_\sigma}_{L^2((0,T)\times M)}\leq C$$
with $C$ a constant independent of $\sigma$. In order to complete the proof, we need to verify that $R_\sigma$ fulfills  \eqref{gGO2}.
For this purpose, we define  
$w_\sigma(t,x)=\int_0^t R_\sigma(\tau,x) d\tau$ 
and observe that $w = w_\sigma$ solves
\bel{eq2_w}
\left\{ \begin{array}{ll}  \pd_t^2 w-\Delta_{g} w  = -H , & \mbox{in}\ (0,T)\times M,\\  
w  =  0, & \mbox{on}\ (0,T)\times \partial M,\\  
 w(0,\cdot)  =0,\quad  \pd_tw(0,\cdot)  =0 & \mbox{in}\ M,
\end{array} \right.
\ee
where $H(t,x) = \int_0^te^{i\sigma s}F_\sigma(s,\cdot)ds + \int_0^tq(s,x)R_\sigma(s,x)ds$.
As both $H$ and $\pd_t H$ are in $L^2((0,T)\times M)$,
we deduce from \cite[Theorem 2.1, Chapter 5]{LM2} that $w \in H^2((0,T)\times M)$. Multiplying equation \eqref{eq2_w} by $\pd_tw$ and integrating 
in $x\in M$  and $s\in(0,t)$ we get
\begin{align*}
\int_0^t\int_M \left(\pd_tw\pd_t^2w-\pd_tw\Delta_gw \right) dV_g ds
=-\int_0^t\int_M H \pd_tw\, dV_g ds,
\end{align*}
where $dV_g$ is the Riemannian volume measure on $(M,g)$.
We define 
$$
\eta_\sigma=\sup_{t\in[0,T]}\norm{\int_0^te^{i\sigma s}F_\sigma(s,\cdot)ds}_{L^2(M)},
$$ 
and obtain, after integration by parts in $x\in M$,
\begin{align}
\label{ll1a}
&\norm{\pd_tw(t,\cdot)}_{L^2(M)}^2+\|\nabla_g w(t,\cdot)\|_{L^2(M)}^2
\leq 
\eta_\sigma\int_0^t\norm{\pd_tw(s,\cdot)}_{L^2(M)}ds
\\\notag&\quad
+ \norm{q}_{L^\infty((0,T) \times M)} \int_0^t\int_0^s\norm{\pd_tw(s,\cdot)}_{L^2(M)} \norm{\pd_tw(\tau,\cdot)}_{L^2(M)} d\tau ds.
\end{align}
Moreover, using the shorthand notation $\zeta(t) = \norm{\pd_tw(t,\cdot)}_{L^2(M)}$, 
\begin{align}
\label{ll1b}
\int_0^t\int_0^s \zeta(s) \zeta(\tau) d\tau ds
&=\int_0^t \zeta(\tau) \int_\tau^t \zeta(s) dsd\tau
= - \frac 1 2 \int_0^t \pd_\tau \left( \int_\tau^t \zeta(s) ds \right)^2 d\tau
\\\notag&
= \frac 1 2 \left( \int_0^t \zeta(s) ds \right)^2.
\end{align}
According to \eqref{ll1a}-\eqref{ll1b}, there is a constant $C > 0$ such that
\begin{align*}
\zeta^2(t) \le \eta_\sigma \int_0^t \zeta(s) ds 
+ C \left( \int_0^t \zeta(s) ds \right)^2
\le \eta_\sigma^2 + (C + 1) \left( \int_0^t \zeta(s) ds \right)^2.
\end{align*} 
Thus, there is a constant $C > 0$ such that
$$
\zeta(t) 
\leq \sqrt{2}\eta_\sigma + C \int_0^t \zeta(s) ds,
$$
and an application of the Gr\"onwall lemma yields 
$$\norm{R_\sigma(t,\cdot)}_{L^2(M)}=\norm{\partial_tw(t,\cdot)} = \zeta(t) \leq \sqrt{2}\eta_\sigma e^{C t},\quad t\in(0,T).$$

It remains to show that $\eta_\sigma \to 0$ as $\sigma \to +\infty$.
According to the Riemann-Lebesgue lemma, for all $t\in(0,T)$ and almost every $x\in M$, we have
$$\lim_{\sigma\to+\infty}\int_0^te^{i\sigma s}F_\sigma(s,x)ds=0.$$
Moreover, by the definition of $F_\sigma$,
$$\abs{\int_0^te^{i\sigma s}F_\sigma(s,x)ds}\leq \int_0^t|(\partial_t^2-\Delta_g+q)a(s,x)| ds,\quad t\in[0,T],\ \ x\in M.$$
Thus, we deduce from Lebesgue's dominated convergence theorem
that 
$$\lim_{\sigma\to+\infty}\norm{\int_0^te^{i\sigma s}F_\sigma(s,\cdot)ds}_{L^2(M)}=0, \quad t \in [0,T]. $$
Combining this with 
\begin{align*}
&\norm{\int_0^{t_2}e^{i\sigma s}F_\sigma(s,\cdot)ds-\int_0^{t_1}e^{i\sigma s}F_\sigma(s,\cdot)ds}_{L^2(M)}
\\&\qquad
\leq (t_2-t_1)\norm{(\partial_t^2-\Delta_g+q)a}_{L^\infty(0,T;L^2(M))},
\quad 0 \leq t_1<t_2 \leq T,
\end{align*}
we deduce that
\bel{ll1c}\lim_{\sigma\to+\infty}\eta_\sigma=\lim_{\sigma\to+\infty}\sup_{t\in[0,T]}\norm{\int_0^te^{i\sigma s}F_\sigma(s,\cdot)ds}_{L^2(M)}=0.\ee
\end{proof}

\subsection{Proof of Theorem \ref{t1full}}

Let $q_j \in L^\infty((0,T) \times M)$, $j=1,2$.
Applying Lemma \ref{Ll1}, for $j=1,2$, we obtain a solution $u_j\in \mathcal H$ of $\pd_t^2u_j-\Delta_gu_j+q_ju_j=0$ having the form 
$$
u_1(t,x)=a_1(t,x)e^{i\sigma(\psi(x)+t)}+R_{1,\sigma}(t,x),\quad u_2(t,x)=a_2(t,x)e^{-i\sigma(\psi(x)+t)}+R_{2,\sigma}.
$$
Here the amplitudes $a_j$, $j=1,2$, are defined in the polar normal coordinates associated to $y\in\pd M_1$
by
\bel{t1c_ampl}a_1(t,r,\theta)=e^{-{\mu (r+t)\over2}}h(\theta)b(r,\theta)^{-1/4},\quad a_2(t,r,\theta)=e^{-{\mu (r+t)\over2}}b(r,\theta)^{-1/4},
\ee
where $h$ is an arbitrary smooth function on the unit sphere at $y$.
Notice that we have chosen $\chi=1$ identically in (\ref{a1}),
and also that the fact that $R_{j,\sigma}$ can be chosen to satisfy vanishing initial conditions (\ref{gGO}) is not used in the full data case. 

We fix $v_1\in\mathcal H$ to be the solution of
\begin{align}
\label{eq_v1}
\left\{ \begin{array}{ll}  \pd_t^2 v_1-\Delta_{g} v_1 + q_1 v_1  =  0, & \mbox{in}\ (0,T)\times M,\\  
v_1  =  u_2, & \mbox{on}\ (0,T)\times \partial M,\\  
 v_1(0,\cdot)  =u_2(0,\cdot),\quad  \pd_t v_1(0,\cdot)  =\pd_t u_2(0,\cdot) & \mbox{in}\ M.
\end{array} \right.
\end{align}
The boundary conditions (\ref{gGO}) for $R_{\sigma, 2}$ imply that $u_2 = a_2 e^{-i\sigma(\psi(x)+t)}$ on the lateral boundary $(0,T)\times \partial M$. In particular, $u_2$ is smooth there and we deduce from \cite{LLT} that $\pd_\nu v_1 \in L^2((0,T)\times \pd M)$. 
We have also $\pd_\nu u_j \in L^2((0,T)\times \pd M)$ by Lemma \ref{Ll1}.
These regularity properties justify  
the integration by parts below.

The function $u=v_1-u_2$ solves
\bel{eq_u}
\left\{ \begin{array}{ll}  \pd_t^2 u-\Delta_{g} u + q_1 u  =  q u_2, & \mbox{in}\ (0,T)\times M,\\  
u  =  0, & \mbox{on}\ (0,T)\times \partial M,\\  
 u(0,\cdot)  =0,\quad  \pd_tu(0,\cdot)  =0 & \mbox{in}\ M,
\end{array} \right.
\ee
where $q = q_2-q_1$.
Moreover, $B_{q_1} = B_{q_2}$ implies that it satisfies also the boundary conditions $\pd_\nu u|_{(0,T) \times \pd M)} = 0$
and $u(T,\cdot) = \pd_t u(T,\cdot) = 0$.
Thus
\begin{align*}
&\int_0^T \int_M q u u_1 dV_g dt 
\\&\quad=
\int_0^T \int_M (\pd_t^2 u-\Delta_{g} u + q_1 u) u_1\, dV_g dt
- \int_0^T \int_M u (\pd_t^2 u_1-\Delta_{g} u_1 + q_1 u_1)\, dV_g dt
\end{align*}
vanishes by integration by parts.
It follows
\[\int_0^T\int_{M}qa_1a_2\, dV_gdt+\int_0^T\int_M Z_\sigma\, dV_gdt=0\]
with $Z_\sigma=q(a_1 R_{2,\sigma} e^{i\sigma(\psi(x)+t)}+a_2R_{1,\sigma}e^{-i\sigma(\psi(x)+t)}+R_{1,\sigma} R_{2,\sigma})$. 
Then in view of \eqref{gGO2} sending $\sigma\to+\infty$ we get
\begin{align}
\label{t1c}
\int_0^{T}\int_{M}qa_1a_2\, dV_gdt=0.
\end{align}

It remains to show that \eqref{t1c} implies $q_1=q_2$.
We extend $q$ by zero to $(0,+\infty) \times M_1$.
Denoting by $\tau_+(y,\theta)$ the time of existence in $M_1$ of the maximal geodesic $\gamma_{y,\theta}$ satisfying $\gamma_{y,\theta}(0)=y$ and $\gamma_{y,\theta}'(0)=\theta$, we obtain in the polar normal coordinates
\begin{align}
\label{vanishing}
\int_0^{+\infty}\int_{S_yM_1}\int_{0}^{\tau_+(y,\theta)}\tilde{q}(t,r,\theta)h(\theta)e^{-{\mu (r+t)}}
\, dr d\theta dt =0,
\end{align}
for all $ h \in C^\infty(S_y M_1)$, $y \in \pd M_1$ and $\mu > 0$.
We used here the fact that $dV_g$ is given by $b(r,\theta)^{1/2}drd\theta$ in the polar normal coordinates. 

The attenuated geodesic ray transform $I_\mu$ 
on the inward pointing boundary of the unit sphere bundle 
$\pd _+SM_1=\{(x,\theta)\in SM_1:\ x\in\pd M_1,\ \left\langle \theta,\nu(x)\right\rangle_g<0\}$ is defined by
\[I_\mu f(x,\theta)=\int_0^{\tau_+(x,\theta)}f(\gamma_{x,\theta}(r))e^{-\mu r}dr,\quad (x,\theta)\in\pd _+SM_1,\ f\in C^\infty(M_1).\]
Here $\mu > 0$ gives constant attenuation. The map 
$I_\mu$ admits a unique continuous extension to the distributions on $M_1$.
We denote by $\mathcal L_\mu$ 
the Laplace transform with respect to $t\in(0,+\infty)$, that is,
\[\mathcal L_\mu f =\int_0^{+\infty} f(t)e^{-\mu t}dt, \quad f\in L^1(0,+\infty). \] 
We see that (\ref{vanishing}) is equivalent with 
$I_\mu \mathcal L_\mu q = 0$, $\mu > 0$, in the sense of distributions on $\pd_+SM_1$.

We deduce from \cite[Proposition 4.1]{FSU} that $\mathcal L_\mu q \in C_0^\infty(M_1)$ for all $\mu > 0$.
Then \cite[Section 7]{DKSU} implies that
there is $\epsilon > 0$ such that $\mathcal L_\mu q = 0$ for $\mu \in (0,\epsilon)$.
Using the fact that $z\mapsto \mathcal L_z q$ is holomorphic in $\{z\in\mathbb C:\ \text{Re} z>0\}$ we see that $\mathcal L_\mu q=0$ for $\mu > 0$. Thus $q=0$.

\subsection{Proof of Theorem \ref{t1}}

Let $q_j \in L^\infty((0,T) \times M)$, $j=1,2$ and let us assume the condition \eqref{t1a} be fulfilled.
Repeating the arguments of Lemma \ref{Ll1}, for $j=1,2$, we obtain a solution $u_j\in \mathcal H$ of $\pd_t^2u_j-\Delta_gu_j+q_ju_j=0$ having the form 
$$ u_1(t,x)=a_1(t,x)e^{-i\sigma(\psi(x)+t)}+R_{1,\sigma}(t,x),
$$
$$
u_2(t,x)=a_2(t,x)e^{i\sigma(\psi(x)+t)}-a_2(-t,x)e^{i\sigma(\psi(x)-t)}+R_{2,\sigma}(t,x).$$
Here the remainder terms $R_{j,\sigma}$, $j=1,2$,
are chosen so that 
$$
R_{1,\sigma}(T,\cdot) = \pd_t R_{1,\sigma}(T,\cdot) = 0, \quad R_{2,\sigma}(0,\cdot)= \pd_t R_{2,\sigma}(0,\cdot) = 0, 
$$
and the amplitudes $a_j$, $j=1,2$, are defined by (\ref{t1c_ampl}).
Note that here $u_2(0,\cdot)=0$.

We fix $v_1\in\mathcal H$ to be again the solution of (\ref{eq_v1})
and set $u=v_1-u_2$. Then $u$ satisfies again (\ref{eq_u}).
Since $v_1(0,\cdot)=u_2(0,\cdot)=0$, the condition $\mathcal C(q_1,0) = \mathcal C(q_2,0)$ implies that $u$ satisfies also the boundary conditions $\pd_\nu u|_{(0,T) \times \pd M)} = 0$
and $u(T,\cdot) = \pd_t u(T,\cdot) = 0$.
Thus the same integration by parts as in the proof of Theorem \ref{t1} gives
\[\int_0^T\int_{M}qa_1a_2\, dV_gdt+\int_0^T\int_M Y_\sigma\, dV_gdt-\int_0^T\left(\int_M q(t,x)a_1(t,x)a_2(-t,x) dV_g\right)e^{-2i\sigma t}dt=0\]
with $$Y_\sigma=q\left[ R_{1,\sigma} u_2+a_1R_{2,\sigma}e^{-i\sigma(\psi(x)+t)}\right].$$ 
Applying the  Riemann-Lebesgue lemma we get
$$\lim_{\sigma\to+\infty}\int_0^T\left(\int_M q(t,x)a_1(t,x)a_2(-t,x) dV_g\right)e^{-2i\sigma t}dt=0$$
and \eqref{gGO2} implies
$$\lim_{\sigma\to+\infty}\int_0^T\int_M Y_\sigma\, dV_gdt=0.$$
Therefore, we get (\ref{t1c}) and, by the proof of Theorem \ref{t1}, it follows that $q=0$.

\subsection{Proof of Theorem \ref{t1var}}

Now let us show that for $T>\textrm{Diam}(M)$, also \eqref{Tt1d} implies \eqref{t1c}. For this purpose, without loss of generality we can assume that $M_1$ is chosen in such a way that $T>\textrm{Diam}(M_1)$ and we consider $u_j$, $j=1,2$, of the form
$$ u_1(t,x)=a_1(t,x)e^{-i\sigma(\psi(x)+t)}-a_1(2T-t,x)e^{-i\sigma(\psi(x)+(2T-t))}+R_{1,\sigma},
$$
$$ u_2(t,x)=a_2(t,x)e^{i\sigma(\psi(x)+t)}-a_2(-t,x)e^{i\sigma(\psi(x)-t)}+R_{2,\sigma},
$$
where the amplitudes $a_j$, $j=1,2$, are now defined in the polar normal coordinates associated to $y\in\pd M_1$
by
$$a_1(t,r,\theta)=e^{-{\mu (r+t)\over2}}\chi(r+t)h(\theta)b(r,\theta)^{-1/4},\quad a_2(t,r,\theta)=e^{-{\mu (r+t)\over2}}\chi(r+t)b(r,\theta)^{-1/4}.
$$
Here, as before, $h$ is an arbitrary smooth function on the unit sphere at $y$, 
and $\chi\in\mathcal C^\infty(\R)$ satisfies $\chi=1$ on a neighborhood of $[0,\textrm{Diam}(M_1)+T]$.
By repeating the above proof once again, we deduce that \eqref{Tt1d} implies
$$\int_0^T \int_M q a_1 a_2 dV_g dt+\int_0^T \int_M qX_\sigma dt+e^{-2i\sigma T}\int_0^T \int_M qa_1(2T-t,x)a_2(-t,x)dV_gdt=0$$
 with $$X_\sigma(t,x)=-a_1(t,x)a_2(-t,x)e^{-2i\sigma t}-a_1(2T-t,x)a_2(t,x)e^{-2i\sigma (T-t)}+R_{2,\sigma}(u_1-R_{1,\sigma})+R_{1,\sigma}u_2(t,x).$$
It follows that
$$\lim_{\sigma\to+\infty}\int_0^T \int_M qX_\sigma dt=0$$
and the expression 
$$\int_0^T \int_M qX_\sigma dt+e^{-2i\sigma T}\int_0^T \int_M qa_1(2T-t,x)a_2(-t,x)dV_gdt$$
admits a limit and  vanishes as $\sigma\to+\infty$ if and only if
$$\int_0^T \int_M qa_1(2T-t,x)a_2(-t,x)dV_gdt=0.$$
This condition will be fulfilled if for all $(t,x)\in (0,T)\times M$ we have $a_1(2T-t,x)a_2(-t,x)=0$. On the other hand, assuming supp$\chi\subset(-\epsilon,\textrm{Diam}(M_1)+T+\epsilon)$ with $2\epsilon<T-\textrm{Diam}(M_1)$, this last condition will be fulfilled and we will deduce \eqref{t1c}. Indeed, note first that
for $r<t-\epsilon$ we have $\chi(r-t)=0$. On the other hand, for $r\geq t-\epsilon$, we have $r+2T-t\geq 2T-\epsilon =T+\textrm{Diam}(M_1)+\epsilon+(T-\textrm{Diam}(M_1)-2\epsilon)>T+\textrm{Diam}(M_1)+\epsilon$. Thus, for $r<t-\epsilon$ we have $\chi(r-t)=0$ and for $r\geq t-\epsilon$ we have $\chi(2T-t+r)=0$. It follows that 
$$\chi(r-t)\chi(2T-t+r)=0,\quad (r,t)\in\R\times\R$$
from which we deduce that
$$a_1(2T-t,x)a_2(-t,x)=0,\quad (t,x)\in (0,T)\times M,$$
and that \eqref{t1c} holds.
The new factor $\chi$ does not cause any changes in the remaining steps of the proof, 
since $\chi=1$ on a neighbourhood of $[0,\textrm{Diam}(M_1)+T]$.

\section{$L^2$-solutions for the direct problem}

From now on we consider the partial data result stated in Theorem \ref{thm1} and we assume that $(M,g)$ satisfies the conditions of this theorem.
In order to perform a duality argument, see Lemma \ref{l3} below, 
we consider solutions to the wave equation that are only in $L^2((0,T) \times M)$.
Let us introduce  the  space
\[H_{\Box_{a,g}}((0,T)\times  M)=\{u\in L^2((0,T)\times  M):\ \Box_{a,g} u \in L^2((0,T)\times  M)\}.\]
It follows from \cite[Theorems B.2.7 and B.2.9]{Ho3}
that the traces 
$$
\p_t^j u|_{t=s}, \quad s = 0,T,\ j=0,1,\ u \in H_{\Box_{a,g}}((0,T)\times  M),
$$
are well-defined as distributions in $\mathcal D'(M)$.
Combining the same argument with a use of boundary normal coordinates, see e.g. \cite[Corollary C.5.3]{Ho3}, 
we see also that the traces
$$
\p_\nu^j u|_{x \in \p M}, \quad j=0,1,\ u \in H_{\Box_{a,g}}((0,T)\times  M),
$$
are well-defined as distributions in $\mathcal D'((0,T) \times \p M)$.

We consider the space
\[S=\{u\in L^2(((0,T)\times  M):\ \Box_{a,g}u=0\}\]
topologized  as a closed subspace of $L^2((0,T)\times  M)$ and 
define the map 
$$
\tau_0 u = (u|_{(0,T) \times \p M}, u|_{\{0\} \times M}, \p_t u|_{\{0\} \times M}), \quad u \in S.
$$
Moreover, we denote by $\mathcal{B}$ the range of $\tau_0$, that is,
$\mathcal{B}=\{\tau_0u:\ u\in S \}$.
The uniqueness of the weak solution of $\Box_{a,g}u=0$, satisfying the vanishing initial and lateral boundary conditions $\tau_0 u = 0$,
implies that the inverse $\tau_0^{-1} : \mathcal{B} \to S$ exists.
We use  $\tau_0^{-1}$ to define a norm on $\mathcal{B}$ by
\[\norm{(f,v_0,v_1)}_{\mathcal{B}}=\norm{\tau_0^{-1}(f,v_0,v_1)}_{L^2((0,T)\times  M)},\quad (f,v_0,v_1)\in\mathcal{B}.\]

Let $U \subset (0,T) \times \p M$ be the set in the definition of the restricted Cauchy data set $\mathcal C_{q,*}$. 
We define the subspace  
\[\mathcal{B}_U=\{(f,v_1) :\ (f,0,v_1)  \in \mathcal{B},\ \supp(f) \subset U\}.\]
We are now ready to show that $\mathcal C_{q,*}$ is a graph.

\begin{prop}\label{p6} Let $(f,v_1)\in \mathcal{B}_U$ and $q\in L^\infty((0,T)\times  M)$. Then the initial boundary value problem
\begin{equation}\label{eq1}\left\{\begin{array}{ll}a^{-1}\partial_t^2u-\Delta_g u+q(t,x)u=0,\quad &\textrm{in}\ (0,T)\times  M,\\  
u=f,\quad &\textrm{on}\ (0,T)\times \pd M,
\\ 
u(0,\cdot)=0,\quad \partial_tu(0,\cdot)=v_1,\quad &\textrm{in}\ M,
\end{array}\right.\end{equation}
admits a unique weak solution $u\in L^2((0,T)\times  M)$ satisfying 
\bel{p6a}
\norm{u}_{L^2((0,T)\times  M)}\leq C\norm{(f,v_1)}_{\mathcal{B}_U}.
\ee
In particular, $\mathcal C_{q,*}$ is the graph of the boundary operator $B_{q,*} (f,v_1) = (\p_\nu u_{|V}, u_{|t=T})$.
\end{prop}
\begin{proof} 
Consider the function $u=v+\tau_0^{-1}(f,0,v_1)$ where $v$ solves
\bel{Eq1}\left\{ \begin{array}{rcll} a^{-1}\partial_t^2v-\Delta_g v+qv& = & -q\tau_0^{-1}(g,0,v_1), & (t,x) \in (0,T)\times  M ,\\ 
v_{\vert(0,T)\times \pd M}& = & 0,&\\
v_{\vert t=0}=\partial_t v_{\vert t=0}&=&0.&
\end{array}\right.
\ee
Since $\tau_0^{-1}(f,0,v_1)\in L^2((0,T)\times  M)$, the equation \eqref{Eq1} admits a unique solution $v\in \mathcal H$, see e.g. Section 8 of Chapter 3 of \cite{LM1}, satisfying
\bel{p6b}\begin{aligned}
\norm{v}_{\mathcal H} &\leq C\norm{-q\tau_0^{-1}(g,0,v_1)}_{L^2((0,T)\times  M)}
\leq C\norm{q}_{L^\infty((0,T)\times  M)}\norm{\tau_0^{-1}(f,v_0,v_1)}_{L^2((0,T)\times  M)}.\end{aligned}\ee
Therefore, $u=v+\tau_0^{-1}(f,0,v_1)$ is the unique solution of \eqref{eq1}, and \eqref{p6b} implies \eqref{p6a}. 
\end{proof}

\section{Carleman estimate}

\begin{thm}\label{c1}  Let $q\in L^\infty((0,T)\times M)$, $\beta\in[1/2,1]$ and $u\in C^2([0,T]\times M)$.  
We use the following notation $s_- = 0$, $s_+ = T$, 
$\psi(x,t) = \beta t + x_1$, $\psi_-(x_1) = -\beta T - x_1$
and $\psi_+(x_1) = x_1$.
If $u$ satisfies the condition 
 \begin{equation}\label{ttc1}u_{\vert (0,T)\times\pd M}=0,\quad u_{\vert t=s_\pm}=\partial_tu_{\vert t=s_\pm}=0,\end{equation}
then there exist constants $\sigma_1>1$ and $C > 0$ depending only on  $M$, $T$ and $\norm{q}_{L^\infty((0,T)\times M)}$ such that the estimate
\begin{equation}\label{c1a}\begin{array}{l}
\sigma \int_M e^{2\sigma \psi_\pm}\abs{\partial_tu(s_\mp,x)}^2dV_g(x)\\
+\sigma\int_{\Sigma_{\mp}}e^{\pm 2\sigma \psi}\abs{\partial_\nu u}^2\abs{\pd_\nu\phi } d\sigma_g(x)dt
+\sigma^2\int_{(0,T)\times M}e^{\pm2\sigma \psi}\abs{u}^2dV_g(x)dt\\
\leq 
C\left(\int_{(0,T)\times M}e^{\pm 2\sigma \psi}\abs{(\Box_{a,g} +q)u}^2dxdt+ \sigma^3\int_M e^{2\sigma \psi_\pm}\abs{u(s_\mp,x)}^2dV_g(x)\right)\\
 +C\left(\sigma\int_M e^{2\sigma\psi_\pm}\abs{\nabla_gu(s_\mp,x)}_{g}^2dV_g(x)
 +\sigma\int_{\Sigma_{\pm}}e^{\pm 2\sigma\psi}\abs{\partial_\nu u}^2\abs{\pd_\nu\phi } d\sigma_g(x)dt\right)\end{array}\end{equation}
holds true for $\sigma\geq \sigma_1$.
\end{thm}

Let us first remark that
\bel{red} (a^{-1}\pd_t^2-\Delta_g+q)\left(a^{-{d-2\over 4}}v\right)=a^{-{d+2\over 4}}\left(\pd_t^2v-\Delta_{e\oplus g'}v+q_av\right),\ee
with $q_a=aq+a^{d+2\over 4}\Delta_g\left(a^{-{d-2\over 4}}\right)$.  
Thus by replacing $q$ with $q_a$, we can assume that $a=1$. From now on, throughout this section, we assume that $a=1$ and consider the leading order wave operator 
$\Box_{e\oplus g'}=\Box_{1,{e\oplus g'}}=\pd_t^2-\Delta_{e\oplus g'}$.

In order to prove the above Carleman estimate, we fix $u\in  C^2(\overline{Q})$ satisfying \eqref{ttc1} 
and we set $v=e^{-\sigma(\beta t+x_1)}u$. Then, fixing  $P_s=e^{-s(\beta t+x_1)}(\partial_t^2-\Delta_{e\oplus g'})e^{s(\beta t+x_1)}$, $s\in\R$,
we get
\begin{equation}\label{c1c}e^{-\sigma(\beta t+x_1)}\Box_g u=P_{\sigma}v
\end{equation}
We begin by proving the following estimate for the conjugated operator $P_{\sigma}$.

\begin{lem}\label{tc} Let $v\in \mathcal C^2([0,T]\times M)$ and $\sigma>1$. If $v$ satisfies the condition
 \begin{equation}\label{tc1}v_{\vert (0,T)\times\pd M}=0,\quad v_{\vert t=0}=\partial_tv_{\vert t=0}=0\end{equation}
then the estimate
 \begin{equation}\label{tc2}\begin{array}{l}{1\over2}\sigma\int_M\abs{\partial_tv(T,x)}^2dV_g(x)+2\sigma\int_{\Sigma_{+}}\abs{\partial_\nu v}^2\pd_\nu\phi  d\sigma_g(x)dt+c\sigma^2\int_{(0,T)\times M}\abs{v}^2dV_g(x)dt\\
\leq \int_{(0,T)\times M}\abs{P_{\sigma}v}^2dV_g(x)dt+ 7\sigma \int_M\abs{\nabla_gv(T,x)}_g^2dV_g(x)+2\sigma\int_{\Sigma_{-}}\abs{\partial_\nu v}^2\abs{\pd_\nu\phi }d\sigma_g(x)dt\\
\ \ \ \ +2\sigma^3\int_M\abs{v(T,x)}^2dV_g(x)\end{array}\end{equation}
holds true for  $c>0$  depending only on  $\beta$ and $T$. 
\end{lem}

\begin{proof} Without loss of generality we assume that $v$ is real valued. 
We fix $v\in C^2([0,T]\times M)$ satisfying \eqref{tc1} and consider
\[I_\sigma=\int_0^T\int_M\abs{e^{-\sigma(\beta t+x_1)}(\partial_t^2-\Delta_g)u}^2dV_g(x)dt.\]
For all $s\in \R$ we decompose $P_s$ into two terms 
$P_s=P_{1,s}+sP_2(\pd_t,\pd_{x_1})$,
with $P_{1,s}=\Box_{e\oplus g'}-(1-\beta^2)s^2$ and $P_2(\pd_t,\pd_{x_1})=2(\beta\pd_t-\pd_{x_1})$. We obtain
\begin{align}
\label{ca3}
I_\sigma &=
\int_0^T\int_M|P_\sigma v|^2dV_g(x)dt
\\\notag&=
\int_0^T\int_M\left[|P_{1,\sigma} v|^2+\sigma^2|P_2(\pd_t,\pd_{x_1})v|^2+2\sigma (P_2(\pd_t,\pd_{x_1})v)(\Box_g v+(\beta^2-1)\sigma^2v)\right]dV_g(x)dt.
\end{align} 
Note first that
\bel{caca}\begin{aligned}\int_0^T\int_M(P_2(\pd_t,\pd_{x_1})v)vdV_g(x)dt&=\beta\int_0^T\int_M\pd_t|v|^2dV_g(x)dt-\int_0^T\int_M\pd_{x_1}|v|^2dV_g(x)dt\\
\ &=\beta \int_0^T\int_M\pd_t|v|^2dV_g(x)dt-\int_0^T\int_M\textrm{div}_g(|v|^2e_1)dV_g(x)dt\\
\ &=\beta\int_M|v|^2(T,x)dV_g(x)\geq0.\end{aligned}\ee
Therefore, we have
\bel{caca1} \int_0^T\int_M(2\sigma P_2(\pd_t,\pd_{x_1})v)((\beta^2-1)\sigma^2v)dV_g(x)dt\geq -2\sigma^3\int_M\abs{v(T,x)}^2dV_g(x).\ee
Moreover, we find
\begin{align*}
&2\sigma \int_0^T\int_M(P_2(\pd_t,\pd_{x_1})v)\Box_g vdV_g(x)dts
\\&\quad
=4\beta\sigma \int_0^T\int_M\pd_tv\Box_gvdV_g(x)dt-4\sigma \int_0^T\int_M\pd_{x_1}v\pd_t^2vdV_g(x)dt+4\sigma \int_0^T\int_M\pd_{x_1}v\Delta_gvdV_g(x)dt
\\&\quad
=I_{1,\sigma}+I_{2,\sigma}+I_{3,\sigma}
\end{align*}
Using the fact that $v_{|(0,T)\times \pd M}=0$, $\pd_tv_{|t=0}=v_{|t=0}=0$ and integrating by parts, we obtain
\[I_{1,\sigma}=2\beta\sigma \int_M (|\pd_tv(T,x)|^2+|\nabla_g v(T,x)|_g^2)dV_g(x).\]
In a same way,   integrating by parts in $t\in(0,T)$ we get
\[\begin{aligned}I_{2,\sigma}&=-4\sigma \int_M \pd_tv(T,x)\pd_{x_1}v(T,x)dV_g(x)+4\sigma\int_0^T\int_M\pd_tv\pd_{x_1}\pd_tvdV_{g}(x)dt\\
\ &=-4\sigma \int_M \pd_tv(T,x)\pd_{x_1}v(T,x)dV_g(x)+2\sigma\int_0^T\int_M\pd_{x_1} |\pd_tv|^2(t,x_1,x')dV_{g}(x)dt\\
\ &=-4\sigma \int_M \pd_tv(T,x)\pd_{x_1}v(T,x)dV_g(x)+2\sigma\int_0^T\int_{M}\textrm{div}_g(|\pd_tv|^2e_1)dV_{g}(x)dt\\
\ &=-4\sigma \int_M \pd_tv(T,x)\pd_{x_1}v(T,x)dV_g(x)\\
\ &\geq-{1\over 2}\sigma \int_M |\pd_tv(T,x)|^2dV_g(x)-8\sigma \int_M |\nabla_g v(T,x)|_g^2dV_g(x) .\end{aligned}\]
Here we have used the fact that $|g|$ is independent of $x_1$ and $\pd_tv_{|(0,T)\times \pd M}=0$. Combining the result for $I_{1,\sigma}$ and $I_{2,\sigma}$ we find
\bel{ca4}\begin{aligned}I_{1,\sigma}+I_{2,\sigma}&\geq (2\beta-{1\over 2})\sigma \int_M |\pd_tv(T,x)|^2dV_g(x)+(2\beta-8)\sigma \int_M |\nabla_g v(T,x)|_g^2dV_g(x)\\
\ &\geq{1\over 2}\sigma \int_M |\pd_tv(T,x)|^2dV_g(x)-7\sigma \int_M |\nabla_g v(T,x)|_g^2dV_g(x).\end{aligned}\ee
For $I_{3,\sigma}$, let us first remark that since $g$ is independent of $x_1$, we have
\[\begin{aligned}2\pd_{x_1}v\Delta_gv&=2\textrm{div}_g(\nabla_gv\pd_{x_1}v)-2\left\langle \nabla_gv,\pd_{x_1}\nabla_gv\right\rangle_g\\
\ &=2\textrm{div}_g(\nabla_gv\pd_{x_1}v)-\pd_{x_1}|\nabla_gv|^2_g\\
\ &=2\textrm{div}_g(\nabla_gv\pd_{x_1}v)-\textrm{div}_g(|\nabla_gv|^2_ge_1).\end{aligned}\]
Using this formula  we get
\[\begin{aligned}I_{3,\sigma}&=4\sigma \int_0^T\int_M\textrm{div}_g(\nabla_gv\pd_{x_1}v)dV_{g}(x)dt-2\sigma \int_0^T\int_M\textrm{div}_g(|\nabla_gv|^2_ge_1)dV_{g}(x)dt\\
\ &=4\sigma \int_0^T\int_{\pd M}\pd_{\nu}v\pd_{x_1}vd\sigma_g(x)dt-2\sigma \int_0^T\int_{\pd M}|\nabla_gv|^2_g\left\langle \nu,e_1\right\rangle_gd\sigma_g(x)dt.\end{aligned}\]
Once again, using the fact that $v_{|(0,T)\times \pd M}=0$, we deduce that $\nabla_gv_{|(0,T)\times \pd M}=(\pd_\nu v)\nu$. Moreover, we have 
\[\pd_{x_1}v_{|(0,T)\times \pd M}={\left\langle \nabla_gv,e_1\right\rangle_g}_{|(0,T)\times \pd M}=\pd_\nu v\left\langle \nu,e_1\right\rangle_g\]
and it follows
\[I_{3,\sigma}=2\sigma \int_0^T\int_{\pd M}|\pd_{\nu}v|^2\left\langle \nu,e_1\right\rangle_gd\sigma_g(x)dt.\]
Combining this with \eqref{ca3}-\eqref{ca4}, we get
 \bel{ca5}\begin{aligned}I_\sigma
\geq \int_0^T\int_M\left[|P_{1,\sigma} v|^2+\sigma^2|P_2(\pd_t,\pd_{x_1})v|^2\right]dV_g(x)dt+2\sigma \int_0^T\int_{\pd M}|\pd_{\nu}v|^2\left\langle \nu,e_1\right\rangle_gd\sigma_g(x)dt\\
+{1\over2}\int_M |\pd_tv(T,x)|^2dV_g(x)-7\sigma \int_M |\nabla_g v(T,x)|_g^2dV_g(x)-2\sigma^3 \int_M | v(T,x)|^2dV_g(x).\end{aligned}\ee
In the same way as \eqref{caca}, for all $\tau\in(0,T)$, we find
\[\int_0^\tau\int_M(P_2(\pd_t,\pd_{x_1})v)vdV_g(x)dt=\beta\int_M|v|^2(\tau,x)dV_g(x).\]
Therefore, an application of the Cauchy-Schwarz inequality yields
\[\beta\int_M|v|^2(\tau,x)dV_g(x)\leq T\int_0^T\int_M|P_2(\pd_t,\pd_{x_1})v|^2dV_g(x)dt+{1\over 4T}\int_0^T\int_M|v|^2dV_g(x)dt.\]
Integrating this inequality with respect to $\tau\in(0,T)$ we get
\[\int_0^T\int_M|v|^2dV_g(x)dt\leq {T^2\over \beta-{1\over4}}\int_0^T\int_M|P_2(\pd_t,\pd_{x_1})v|^2dV_g(x)dt.\]
Combining this with  \eqref{ca5}, and observing that $g=e\oplus g'$ implies $\pd_\nu\phi=\left\langle e_1,\nu\right\rangle_g$,  we get \eqref{tc2}. 
\end{proof}

\begin{proof}[Proof of Theorem \ref{c1}]
Let us first consider the case $q=0$. Note  that for $u$ satisfying \eqref{ttc1} with the minus sign, $v=e^{-\sigma(\beta t+x_1)}u$ satisfies  \eqref{tc1}. Moreover, \eqref{ttc1}
and  \eqref{c1c} imply $$\partial_\nu v_{\vert(0,T)\times\pd M}=e^{-\sigma(\beta t+x_1)}\partial_\nu u_{\vert(0,T)\times\pd M}.$$ Finally, using the fact that
\[\partial_tu=\partial_t(e^{\sigma(\beta t+x_1)} v)=\beta\sigma u+e^{\sigma(\beta t+x_1)} \partial_tv,\quad \nabla_g v=e^{-\sigma(\beta t+x_1 )}(\nabla_g u-\sigma ue_1),\]
we obtain
\[\int_M e^{-2\sigma(\beta T+x_1)}\abs{\partial_tu(T,x)}^2dV_g(x)\leq 2\int_M \abs{\partial_tv(T,x)}^2dV_g(x)+2\sigma^2\int_M e^{-2\sigma(\beta T+x_1)}\abs{u(T,x)}^2dV_g(x),\]
\[\int_M\abs{\nabla_g v(T,x)}^2dV_g(x)\leq 2\sigma^2\int_M e^{-2\sigma(\beta T+x_1)}\abs{u(T,x)}^2dV_g(x)+2\int_M e^{-2\sigma(\beta T+x_1)}\abs{\nabla_g u(T,x)}^2dV_g(x).\]
 Thus, applying the Carleman estimate \eqref{tc2} to $v$, we deduce \eqref{c1a}. For $q\neq0$,
we have 
 \[\abs{\partial_t^2u-\Delta_gu }^2=\abs{\partial_t^2u-\Delta_gu+qu-qu}^2\leq 2\abs{(\partial_t^2-\Delta_g+q)u}^2+2\norm{q}^2_{L^\infty((0,T)\times M)}\abs{u}^2\]
 and hence if we choose $\sigma_1>2C\norm{q}^2_{L^\infty((0,T)\times M)}$, replacing $C$ by
 \[C_1=\frac{C\sigma_1^2}{\sigma_1^2-2C\norm{q}^2_{L^\infty((0,T)\times M)}},\]
 we deduce \eqref{c1a}  from the same estimate when $q=0$. 
 
The case with plus sign in \eqref{ttc1} is analogous. 
Note that in this case we can apply Lemma \ref{tc} after the time reversal $t \mapsto T - t$.
\end{proof}

\begin{rmk}\label{rr} Note that, by density, estimate \eqref{c1a} can be extended to any function $u\in\mathcal H$ satisfying \eqref{tc1}, $(a^{-1}\partial_t^2-\Delta_g)u\in L^2((0,T)\times M)$ and $\partial_\nu u\in L^2((0,T)\times \pd M)$.\end{rmk}

\section{Exponentially decaying and  growing geometric optics solutions}

This section starts with a construction of exponentially decaying solutions $u_2\in H^1((0,T)\times M)$   taking the form
\bel{dec}u_1(t,x)=e^{-\sigma(\beta t+\phi(x))}(a_{1,\sigma}(t,x)+R_1(t,x)),\ee
where $\beta\in(0,1]$, $\sigma>0$.
Then we consider exponentially growing solutions $u_2\in L^2((0,T)\times M)$   taking the form
\bel{gro}u_2(t,x)=e^{\sigma(\beta t+\phi(x))}(a_{2,\sigma}(t,x)+R_2(t,x))\ee
and satisfying the additional condition 
$$u_2(t,x)=0,\quad  (t,x)\in(\{0\}\times M)\cup U.$$
Here $R_j$, $j=1,2$, denotes the remainder term in the expression of the solution $u_j$ with respect to the parameter $\sigma$ in such a way that there exists $\gamma\in (0,1)$ such that  $\norm{R_j}_{L^2(Q)}\leq C\sigma^{-\gamma}$. 
We give different arguments in the two cases. 
For the  exponentially decaying solutions  $ u_1\in H^1((0,T)\times M)$, we combine an argument of separation of variables with properties of solutions of PDEs with constant coefficients. For the exponentially growing solutions $u_2$, inspired by \cite{KSU}, we apply the Carleman estimate \eqref{c1a} and the Hahn-Banach theorem to obtain these solutions by duality. 
Combining these two types of solutions we derive Theorem \ref{thm1} in the next section.

\subsection{Exponentially decaying solutions without boundary conditions }

We extend our manifold  $M$ into a cylindrical manifold and we will consider the restriction on $(0,T)\times M$ of exponentially decaying solutions on the extended domain. More precisely, we first fix $R>0$, $M_1\subset int (M_0)$ a simple manifold such that $M\subset (-R/2,R/2)\times M_1$ and we extend $q$ by zero to a function 
lying in $ L^\infty ((0,T)\times(-R,R)\times M_0)$. Then, in view of \eqref{red}, we consider $q_a=aq+a^{d+2\over 4}\Delta_g\left(a^{-{d-2\over 4}}\right)$ and $u=a^{-{d-2\over 4}}v$ where, for $\sigma>1$, $\beta\in[1/2,1]$, $v$ is a solution of
\begin{equation}\label{eqqGO2}\partial_t^2v-\Delta_{e\oplus g'}v+q_a(t,x)v=0\quad \textrm{on } (0,T)\times(-R,R)\times M_1\end{equation}
  taking the form
\begin{equation}\label{GO1} v(t,x)=e^{-\sigma(\beta t+x_1)}(k(t,x)+w(t,x)),\quad (t,x)\in (0,T)\times(-R,R)\times M_1\end{equation}
with $w\in H^1((0,T)\times(-R,R)\times M_1)$ satisfying
\[\norm{w}_{L^2((0,T)\times(-R,R)\times M_1)}\leq \frac{C}{\sigma}.\]
We extend $(M_0,g')$ to a slightly larger simple manifold $(D,g')$,
and define $k(t,x)$ in polar normal coordinate associated to  $y\in \pd D$ in the following way: we fix  $h\in H^2(S_y(D))$ and we consider  $k(t,x_1,r,\theta)=k(t,x_1,\exp_y(r\theta))$ defined  on $(0,T)\times(-R,R)\times \exp_y^{-1}M_0$ by
\begin{equation}\label{GO2} k(t,x_1,r,\theta)=e^{i\sigma(\sqrt{1-\beta^2}) r}e^{-i\mu(t+\beta x_1)}b(r,\theta)^{-1/4}h(\theta),\end{equation}
where $\mu\in \R$ is arbitrary fixed. It is clear that $v$ solves \eqref{eqqGO2} if and only if $w$ solves
\begin{equation}\label{eqGO2} P_{-\sigma}w=-q_aw-e^{\sigma(\beta t+x_1)}(\Box_{e\oplus g'}+q_a)e^{-\sigma(\beta t+x_1)}k(t,x)\ee
with $P_s$, $s\in\R$, the conjugated operator introduced in Section 4. To find a suitable solution of this equation we consider first  equations   of the form
\begin{equation}\label{eqGO3} P_{-\sigma}y=F,\quad (t,x)\in(0,T)\times(-R,R)\times M_1 .\ee
Consider the selfadjoint operator $A=-\Delta_{g'}$ defined as an unbounded operator on $L^2(M_0)$ with domain $D(A)=H^2(M_0)\cap H^1_0(M_0)$. It is well known that the spectrum of $A$ consist of a non decreasing sequence of positive eigenvalues $(\lambda_n)_{n\geq1}$ associated to an Hilbertian basis of eigenfunctions $(\phi_n)_{n\geq1}$.  Extending $F$ to $(0,T)\times(-R,R)\times M_0$, fixing $n\geq1$  and projecting equation \eqref{eqGO3} on the space spanned by $\phi_n$, we obtain
\begin{equation}\label{eqGO4} P_{n,-\sigma}y=F_n\ee
with $F_n(t,x_1)=\left\langle F(t,x_1,\cdot),\phi_n\right\rangle_{L^2(M_0)}$ and $$P_{n,-\sigma}=\partial_t^2-\partial_{x_1}^2-2\sigma(\beta\pd_t-\pd_{x_1})-(1-\beta^2)\sigma^2+\lambda_n.$$
We set also $p_{n,-\sigma}(\mu,\eta)=-\mu^2+\eta^2-2i\sigma(\beta\mu-\eta)-(1-\beta^2)\sigma^2+\lambda_n$, $\mu\in\R$, $\eta\in\R$, such that, for $D_t=-i\partial_t$, $D_{x_1}=-i\pd_{x_1}$, we have $p_{n,-\sigma}(D_t,D_{x_1})=P_{n,-\sigma}$. Applying some  results of \cite{Ch,Ho2,Ki2,Ki3} about solutions of  PDEs with constant coefficients we obtain the following.

\begin{lem}\label{l1} For every $\sigma>1$ and $n\geq1$  there exists a bounded operator $$E_{n,\sigma}:\ L^2((0,T)\times(-R,R))\to L^2((0,T)\times(-R,R))$$ such that:
\begin{equation}\label{l1a}P_{n,-\sigma} E_{n,\sigma}F=F,\quad F\in L^2((0,T)\times(-R,R)),\end{equation}
\begin{equation}\label{l1b} \norm{E_{n,\sigma}}_{\mathcal B(L^2((0,T)\times(-R,R)))}\leq C\sigma^{-1},\end{equation}
\begin{equation}\label{l1c} E_{n,\sigma}\in \mathcal B(L^2((0,T)\times(-R,R));H^1((0,T)\times(-R,R)))\ee and
\bel{l1d}\quad \norm{E_{n,\sigma}}_{\mathcal B(L^2((0,T)\times(-R,R));H^1((0,T)\times(-R,R)))}+\norm{\sqrt{\lambda_n}E_{n,\sigma}}_{\mathcal B(L^2((0,T)\times(-R,R))}\leq C\end{equation}
with $C>$ depending only on  $T$ and $R$.\end{lem}
\begin{proof} In light of \cite[Thoerem 2.3]{Ch} (see also \cite[Theorem 10.3.7]{Ho2}), there exists a bounded operator $E_{n,\sigma}\in \mathcal B( L^2((0,T)\times(-R,R)))$, defined from a fundamental solutions associated to $P_{n,-\sigma}$ (see Section 10.3 of \cite{Ho2}),  such that \eqref{l1a} is fulfilled. In addition, fixing 
\[\tilde{p}_{n,-\sigma}(\mu,\eta):=\left(\sum_{k\in\mathbb N}\sum_{\alpha\in\mathbb N}|\partial^k_\mu\partial^\alpha_\eta p_{n,-\sigma}(\mu,\eta)|^2\right)^{{1\over2}},\quad \mu\in\R,\ \eta\in\R,\]
for all differential operator  $Q(D_t,D_{x_1})$ with ${Q(\mu,\eta)\over \tilde {p}_{n,-\sigma}(\mu,\eta)}$ a bounded function, we have $Q(D_t,D_{x_1})E_{n,\sigma}\in\mathcal B(L^2((0,T)\times(-R,R)))$ and there exists  a constant $C$ depending only on $R$, $T$ such that
\begin{equation}\label{l1e}\norm{Q(D_t,D_x)E_{\sigma}}_{\mathcal B(L^2((0,T)\times(-R,R)))}\leq C\sup_{(\mu,\eta)\in\R^{2}}{|Q(\mu,\eta)|\over \tilde {p}_{n,-\sigma}(\mu,\eta)}.\end{equation}
Note that $\tilde{p}_{n,-\sigma}(\mu,\eta)\geq \abs{\im \partial_\eta p_{n,-\sigma}(\mu,\eta)}=2\sigma$. Therefore, \eqref{l1e} implies
\[\norm{E_{\sigma}}_{\mathcal B(L^2((0,T)\times(-R,R)))}\leq C\sup_{(\mu,\eta)\in\R^{2}}{1\over \tilde {p}_{n,-\sigma}(\mu,\eta)}\leq C\sigma^{-1}\]
and \eqref{l1b} is fulfilled. In a same way, we have $\tilde{p}_{-\sigma}(\mu,\eta)\geq \abs{\re \partial_\mu p_{-\sigma}(\mu,\eta)}=2|\mu|$ and $\tilde{p}_{n,-\sigma}(\mu,\eta)\geq \abs{\re \partial_{\eta} p_{n,-\sigma}(\mu,\eta)}=2|\eta|$. Therefore, in view of  \cite[Theorem 2.3]{Ch}, we have \eqref{l1c} with
\[\norm{E_{n,-\sigma}}_{\mathcal B(L^2((0,T)\times(-R,R));H^1((0,T)\times(-R,R)))}\leq C\sup_{(\mu,\eta)\in\R^{2}}{|\mu|+|\eta|\over \tilde {p}_{n,-\sigma}(\mu,\eta)}+C\sigma^{-1}\leq 3C\]
and the first inequality of \eqref{l1d} is proved. For the last inequality of \eqref{l1d}, let us consider the two cases $\lambda_n\geq 2(|\mu|+\sigma)^2$ and $\lambda_n< 2(|\mu|+\sigma)^2$. For $\lambda_n\geq2(|\mu|+\sigma)^2$, we have
\bel{l1f}\tilde{p}_{n,-\sigma}(\mu,\eta) \geq \abs{\re  p_{n,-\sigma}(\mu,\eta)}\geq -\mu^2-(1-\beta^2)\sigma^2+\lambda_n\geq \lambda_n-(|\mu|+\sigma)^2={\lambda_n\over2}\geq c\sqrt{\lambda_n}\ee
with $c>0$ independent of $n$. For $\lambda_n<2(|\mu|+\sigma)^2$, we have $|\mu|>{\sqrt{\lambda_n}\over\sqrt{2}}-\sigma$ and we get
\bel{l1g}\tilde{p}_{n,-\sigma}(\mu,\eta) \geq {(\abs{\re  \pd_\mu p_{n,-\sigma}(\mu,\eta)}+\abs{\im  \pd_\eta p_{n,-\sigma}(\mu,\eta)})\over2}\geq |\mu|+\sigma\geq {\sqrt{\lambda_n}\over\sqrt{2}}.\ee
Combining \eqref{l1e}-\eqref{l1g}, we deduce the second inequality of \eqref{l1d}.\end{proof}
Applying this lemma, we can now consider solutions of \eqref{eqGO3} given by the following result.
\begin{lem}\label{l2} For every $\sigma>1$ and $n\geq1$  there exists a bounded operator $$E_{\sigma}:\ L^2((0,T)\times(-R,R))\times M_0)\to L^2((0,T)\times(-R,R)\times M_0)$$ such that:
\begin{equation}\label{l2a}P_{-\sigma} E_{\sigma}F=F,\quad F\in L^2((0,T)\times(-R,R)\times M_0),\end{equation}
\begin{equation}\label{l2b} \norm{E_{\sigma}}_{\mathcal B(L^2((0,T)\times(-R,R)\times M_0))}\leq C\sigma^{-1},\end{equation}
\begin{equation}\label{l2c} E_{\sigma}\in \mathcal B(L^2((0,T)\times(-R,R)\times M_0);H^1((0,T)\times(-R,R)\times M_0))\ee and
\bel{l2d}\quad \norm{E_{\sigma}}_{\mathcal B(L^2((0,T)\times(-R,R)\times M_0);H^1((0,T)\times(-R,R)\times M_0))}\leq C\end{equation}
with $C>$ depending only on  $T$, $R$ and $M_0$.\end{lem}
\begin{proof}
According to Lemma \ref{l1}, we can define $E_\sigma$ on $L^2((0,T)\times(-R,R)\times M_0)$  by
\[E_{\sigma}F:=\sum_{n=1}^\infty (E_{n,\sigma}F_n)\phi_n,\quad F_n(t,x_1)=\left\langle F(t,x_1,\cdot),\phi_n\right\rangle_{L^2(M_0)},\ (t,x_1)\in\times(0,T)\times(-R,R).\]
It is clear that \eqref{l1a} implies \eqref{l2a}. Moreover, we have
\[\norm{E_{\sigma}F}^2_{L^2((0,T)\times(-R,R)\times M_0)}=\sum_{n=1}^\infty\norm{E_{n,\sigma}F_n}^2_{L^2((0,T)\times(-R,R))}\]
and  from \eqref{l1b} we get 
\[\norm{E_{\sigma}F}^2_{L^2((0,T)\times(-R,R)\times M_0)}\leq C^2\sigma^{-2}\sum_{n=1}^\infty\norm{F_n}^2_{L^2((0,T)\times(-R,R))}=C^2\sigma^{-2}\norm{F}^2_{L^2((0,T)\times(-R,R)\times M_0)}.\]
From this estimate we deduce \eqref{l2b}. In view of \eqref{l1c}-\eqref{l1d}, we have $E_{\sigma}\in \mathcal B(L^2((0,T)\times(-R,R)\times M_0);H^1((0,T)\times(-R,R); L^2( M_0)))$ and, for all $F\in L^2((0,T)\times(-R,R)\times M_0)$,  we have
\bel{l2e}\begin{aligned}\norm{E_{\sigma}F}_{ H^1((0,T)\times(-R,R); L^2( M_0))}^2&=\sum_{n=1}^{\infty}\norm{E_{n,\sigma}F_n}_{ H^1((0,T)\times(-R,R)}^2\\
\ &\leq C^2\sum_{n=1}^\infty\norm{F_n}^2_{L^2((0,T)\times(-R,R))}=C^2\norm{F}^2_{L^2((0,T)\times(-R,R)\times M_0)}.\end{aligned}\ee
In the same way according to \eqref{l1d}, for all $F\in L^2((0,T)\times(-R,R)\times M_0)$, we have
\[\sum_{n=1}^\infty \lambda_n\norm{\left\langle E_{\sigma}F,\phi_n\right\rangle_{L^2(M_0)}}_{L^2((0,T)\times(-R,R))}^2\leq C^2\sum_{n=1}^\infty\norm{F_n}^2_{L^2((0,T)\times(-R,R))}=C^2\norm{F}^2_{L^2((0,T)\times(-R,R)\times M_0)}.\]
Therefore, we have $E_{\sigma}F\in L^2((0,T)\times(-R,R); D(A^{1/2}))=L^2((0,T)\times(-R,R);H^1_0(M_0))$ and we get
\[\begin{aligned}\norm{E_{\sigma}F}_{L^2((0,T)\times(-R,R);H^1( M_0))}^2&\leq C'\sum_{n=1}^{\infty}\lambda_n\norm{E_{n,\sigma}F_n}_{L^2((0,T)\times(-R,R))}^2\\
\ &\leq C'C^2\sum_{n=1}^\infty\norm{F_n}^2_{L^2((0,T)\times(-R,R))}\\
\ &\leq C'C^2\norm{F}^2_{L^2((0,T)\times(-R,R)\times M_0)}.\end{aligned}\]
Combining this estimate with \eqref{l2e} we deduce \eqref{l2c}-\eqref{l2d}.\end{proof}

Applying this result, we can build  geometric optics solutions of the form \eqref{GO1}.
  \begin{prop}\label{p2} Let $q\in L^\infty((0,T)\times (-R,R)\times M_0)$.  Then, there exists $\sigma_0>1$ such that for $\sigma\geq \sigma_0$ the equation $\Box_{a,g} u + q u = 0$ admits a solution $u\in H^1((0,T)\times (-R,R)\times M_0)$ of the form \eqref{GO1} with
\begin{equation}\label{p2a}\norm{w}_{H^k((0,T)\times (-R,R)\times M_0)}\leq C\sigma^{k-1},\quad k=0,1,\end{equation}
where $C$ and $\sigma_0$ depend on  $M_0$, $T$, $ \norm{q}_{L^{\infty}((0,T)\times (-R,R)\times M_0)}$.
\end{prop}
\begin{proof} We start by recalling that, in view of \eqref{GO2}, in the polar normal coordinate $x'=\exp(r\theta)$ the function $\Box_{g}(e^{-\sigma(\beta t+x_1)}k)$ will be given by
\[\begin{array}{l}\Box_g(e^{-\sigma(\beta t+x_1)}k)(t,x_1,\exp_y(r\theta))=
\sigma^2(\beta^2-1+1-\beta^2)e^{-\sigma(\beta t+x_1)} k\\
+2\sigma e^{-\sigma(\beta t+x_1-i(1-\beta^2)^{1\over2} r)}(-\beta\partial_t+\pd_{x_1}-i\sqrt{1-\beta^2}\pd_r-i\sqrt{1-\beta^2}{\pd_rb\over 4b})e^{-i\mu(t+\beta x_1)}b(r,\theta)^{-1/4}h(\theta)\\
+e^{-\sigma(\beta t+x_1-i(1-\beta^2)^{1\over2} r)}\Box_{e\oplus\tilde{g'}}e^{-i\mu(t+\beta x_1)}b(r,\theta)^{-1/4}h(\theta).\end{array}\]
Using the definition of $ k$ in polar normal coordinate we deduce that
\[(-\beta \partial_t+\pd_{x_1}-i\sqrt{1-\beta^2}\pd_r-i\sqrt{1-\beta^2}{\pd_rb\over 4b})e^{-i\mu(t+\beta x_1)}b(r,\theta)^{-1/4}h(\theta)=0\]
and
\[e^{\sigma(\beta t+x_1)}\Box_g(e^{-\sigma(\beta t+x_1)}k)(t,x_1,\exp_y(r\theta))=e^{i\sigma(\sqrt{1-\beta^2}) r}\Box_{e\oplus\tilde{g'}}e^{-i\mu(t+\beta x_1)}b(r,\theta)^{-1/4}h(\theta).\]
Thus,  there exists $C>0$ independent of $\sigma$ such that
\bel{p2b}\norm{e^{\sigma(\beta t+x_1)}\Box_g(e^{-\sigma(\beta t+x_1)}k)}_{L^2((0,T)\times (-R,R)\times M_0)}\leq C.\ee
According to Lemma \ref{l2}, we can rewrite equation \eqref{eqGO2} as $$w=-E_{\sigma}\left(e^{\sigma(\beta t+x_1)}\Box_g(e^{-\sigma(\beta t+x_1)}k)+qw\right),\quad w\in L^2((0,T)\times (-R,R)\times M_0)$$ with $E_{\sigma}\in\mathcal B(L^2((0,T)\times (-R,R)\times M_0))$ given by Lemma \ref{l2}.
For this purpose, we will use a standard fixed point argument associated to the map
\begin{align*}
\mathcal G:  L^2((0,T)\times (-R,R)\times M_0) & \to L^2((0,T)\times (-R,R)\times M_0), 
\\ 
F &\mapsto -E_{\sigma}\left[e^{\sigma(\beta t+x_1)}\Box_g(e^{-\sigma(\beta t+x_1)}k)+qF\right].
\end{align*}
Indeed, in view of \eqref{l2b}, fixing $R_1>0$ , there  exists $\sigma_0>1$ such that for $\sigma\geq \sigma_0$ the map $\mathcal G$ admits a unique fixed point $w$ in $\{u\in L^2((0,T)\times (-R,R)\times M_0): \norm{u}_{L^2((0,T)\times (-R,R)\times M_0)}\leq R_1\}$. In addition, condition \eqref{l2b}-\eqref{l2d} imply that $w\in H^1((0,T)\times (-R,R)\times M_0)$ fulfills \eqref{p2a}. This completes the proof.
\end{proof}

\subsection{Exponentially growing solutions vanishing on parts of the boundary}
Let us first remark that repeating the arguments of the previous section we can build solutions
\bel{CGO1a}
u(t,x)=e^{\sigma (\beta t+ x_1)}\left( a^{-{d-2\over4}}l(x)+z(t,x) \right),\quad (t,x)\in (0,T)\times M
\ee
of the equation $a^{-1}\partial_t^2u-\Delta_gu +qu=0$ in $(0,T)\times M$. On the other hand, since the construction of the previous section consists of extending the domain  and considering restriction of the solutions of our wave equation on the extended domain, we will have no control on the traces of the solutions on $\pd\overline{M}$ (with $\overline{M}=[0,T]\times M$). On the other hand, according to \cite{KSU,Ki2,Ki3}, one can use Carleman estimates to construct by duality such solutions vanishing on some parts of $\pd\overline{M}$. Following this idea, in this section we will construct solutions of the form \eqref{CGO1a} satisfying 
$$u(t,x)=0,\quad  (t,x)\in(\{0\}\times M)\cup U.$$
 From now on, for  all  $\delta>0$, we set
\[\partial M_{+,\delta,\pm}=\{x\in\partial M:\ \pm\pd_\nu\phi(x)>\delta\},\quad\partial M_{-,\delta,\pm}=\{x\in\partial M:\ \pm\pd_\nu\phi(x)\leq \delta\}\]
and $\Sigma_{\pm,\delta,\pm}=(0,T)\times \partial M_{\pm,\delta,\pm}$.  Without loss of generality we  assume that there exists $0<\epsilon<1$ such that  $\partial M_{-,\epsilon,-}\subset U'$.
The goal of this section is to use the Carleman estimate \eqref{c1a} in order to build  solutions $u\in H_{\Box_{a,g}}((0,T)\times M)$  of the form \eqref{CGO1a} to 
\begin{equation}
\label{(5.1)}
\left\{
\begin{array}{l}
(a^{-1}\partial_t^2-\Delta_g +q(t,x))u=0\ \ \textrm{in }  (0,T)\times M,
\\
u_{\vert t=0}=0,
\\
u=0,\quad \ \textrm{on } (0,T)\times\pd M_{+,\epsilon/2,-}.
\end{array}
\right.
\end{equation}
Here $l\in C^\infty([-R,R]\times M_0)$ is defined in polar normal coordinate associated to $y\in\pd D$ by
\bel{GGO}l(x_1,\exp_y(r\theta))=e^{-i\sigma(\sqrt{1-\beta^2}) r}b(r,\theta)^{-1/4}.\ee
Moreover  $z \in e^{-\sigma (\beta t+x_1)}H_{\Box_{a,g}}((0,T)\times M)$ fulfills: $z(0,x)=-a^{-{d-2\over4}}l(x)$ , $x\in M$, $z=-a^{-{d-2\over4}}l$ on $(0,T)\times\pd M_{+,\epsilon/2,-}$ and
\bel{CGO1b}
\| z \|_{L^2((0,T)\times M)}\leq C\sigma^{-\frac{1}{2}}
\ee
with $C$ depending on $U'$, $M$, $T$ and any $M\geq\norm{q}_{L^\infty((0,T)\times M)}$. Since $(0,T)\times \pd M\setminus U\subset(0,T)\times\pd M\setminus \Sigma_{-,\epsilon,-}=\Sigma_{+,\epsilon,-}$ and since $\Sigma_{+,\epsilon/2,-}$ is a neighborhood of $\Sigma_{+,\epsilon,-}$ in $(0,T)\times \pd M$,  it is clear that condition \eqref{(5.1)} implies $(u_{|(0,T)\times \pd M},\partial_t u_{|t=0})\in\mathcal{B}_U$.

The main result of this section can be stated as follows.

\begin{thm}\label{t3} Let $q\in L^\infty((0,T)\times M)$. For all $\sigma\geq \sigma_1$, with $\sigma_1$ the constant of Theorem \ref{c1},  there exists a solution $u\in H_{\Box_{a,g}}((0,T)\times M)$ of \eqref{(5.1)} of the form \eqref{CGO1a} with $z$ satisfying \eqref{CGO1b}. \end{thm}

In order to prove existence of such solutions of \eqref{(5.1)} we need some weighted spaces. We set $s\in\R$ and we introduce  the spaces $L_s((0,T)\times M)$, $L_s(M)$,  and for all non negative measurable function $h$ on $\partial M$ the spaces $L_{s,h,\pm}$  defined respectively by
\begin{align*}
&L_s((0,T)\times M)=e^{-s(\beta t+ \phi(x))}L^2((0,T)\times M),\quad L_s(M)=e^{-s\phi(x)}L^2(M),
\\&L_{s,h,\pm}=\{f:\ e^{s(\beta t+ x_1)}h^{{1\over2}}(x)f\in L^2(\Sigma_{\pm})\},
\end{align*}
with the associated norms
\[\norm{u}_s=\left(\int_0^T\int_Me^{2s(\beta t+x_1)}\abs{u}^2dV_g(x)dt\right)^{\frac{1}{2}},\quad u\in L_s(Q),\]
\[\norm{u}_{s,0}=\left(\int_M e^{2sx_1}\abs{u}^2dx\right)^{\frac{1}{2}},\quad u\in L_s(M),\]
\[\norm{u}_{s,h,\pm}=\left(\int_{\Sigma_{\pm}} e^{2s(\beta t+x_1)}h(x)\abs{u}^2d\sigma_g(x)dt\right)^{\frac{1}{2}},\quad u\in L_{s,h,\pm}.\]

Combining  the Carleman estimate \eqref{c1a}  with the arguments used in  \cite[Lemma 3]{Ki2}, we obtain the following

\begin{lem}\label{l3}  Given $\sigma\geq \sigma_1$, with $\sigma_1$ the constant of Theorem \ref{c1}, and
\[v\in L_{-\sigma}((0,T)\times M),\quad v_-\in L_{-\sigma,\pd_\nu \phi^{-1},-},\quad v_0\in L_{-\sigma}(M),\]
there exists  $u\in L_{-\sigma}((0,T)\times M)$ such that:\\
1) $(a^{-1}\partial_t^2-\Delta_g+q)u=v$ in $(0,T)\times M$,\\
2) $u_{\vert \Sigma_{-}}=v_-,\ u_{|t=0}=v_0$,\\
3) $\norm{u}_{-\sigma}\leq  C\left(\sigma^{-1}\norm{v}_{-\sigma}+\sigma^{-\frac{1}{2}}\norm{v_-}_{-\sigma,\pd_\nu\phi^{-1},-}+\sigma^{-\frac{1}{2}}\norm{v_0}_{-\sigma,0}\right)$ with $C$ depending on  $T$,\\
 $M\geq\norm{q}_{L^\infty((0,T)\times M)}$.
\end{lem}

Armed with this lemma we are now in position to prove Theorem \ref{t3}.
\begin{proof}[Proof of Theorem \ref{t3}]
Recall that in the polar normal coordinate $x'=\exp_y(r\theta)$ the function $\Box_{e\oplus g'}(e^{\sigma(\beta t+x_1)}l)$ will be given by
\[\begin{array}{l}\Box_{e\oplus g'}(e^{\sigma(\beta t+x_1)}l)(t,x_1,\exp_y(r\theta))=
\sigma^2(\beta^2-1+1-\beta^2)e^{\sigma(\beta t+x_1)} l\\
+2\sigma e^{\sigma(\beta t+x_1-i(1-\beta^2)^{1\over2} r)}(\beta\partial_t-\pd_{x_1}+i\sqrt{1-\beta^2}\pd_r+i\sqrt{1-\beta^2}{\pd_rb\over 4b})b(r,\theta)^{-1/4}\\
+e^{\sigma(\beta t+x_1-i(1-\beta^2)^{1\over2} r)}\Box_{e\oplus\tilde{g'}}b(r,\theta)^{-1/4}.\end{array}\]
Using the definition of $l$ in polar normal coordinate we deduce that
\[(b\partial_t-\pd_{x_1}+i\sqrt{1-\beta^2}\pd_r+i\sqrt{1-\beta^2}{\pd_rb\over 4b})b(r,\theta)^{-1/4}=0\]
and
\[e^{-\sigma(\beta t+x_1)}\Box_{e\oplus g'}(e^{\sigma(\beta t+x_1)}l)(t,x_1,\exp_y(r\theta))=e^{-i\sigma(\sqrt{1-\beta^2}) r}\Box_{e\oplus\tilde{g'}}b(r,\theta)^{-1/4}.\]
Therefore, we have
$$\norm{\Box_{e\oplus g'}(e^{\sigma(\beta t+x_1)}l)}_{-\sigma}\leq C$$
with $C>0$ independent of $\sigma$. Combining this estimate with \eqref{red}, we deduce that
\bel{t3d}\norm{\Box_{a,g}(e^{\sigma(\beta t+x_1)}a^{-{d-2\over 4}}l)}_{-\sigma}\leq C.\ee
Note that $z$ must satisfy
\begin{equation}
\label{w1}
\left\{
\begin{array}{l}z\in L^2((0,T)\times M), \\
(a^{-1}\partial_t^2-\Delta_g+q) (e^{\sigma(\beta t+x_1)}z)=-(\Box_{a,g}+q)(e^{\sigma(\beta t+x_1)}a^{-{d-2\over 4}}l)\ \ \textrm{in }(0,T)\times M,
\\
z(0,x)=-a^{-{d-2\over4}}l(x), \quad   x\in M,
\\
z(t,x)=-a^{-{d-2\over4}}l(x)\quad (t,x)\in\Sigma_{+,\epsilon/2,-}.
\end{array}
\right.\end{equation}
Let $\chi\in\mathcal C^\infty_0(M_0)$ be such that   supp$\chi\cap\partial M\subset \{x\in\partial M:\ \pd_\nu\phi(x)<-\epsilon/3\}$ and $\chi=1$ on $\{x\in\pd M:\ \pd_\nu\phi(x)<-\epsilon/2\}=\partial M_{+,\epsilon/2,-}$. Choose $v_-(t,x)=-e^{\sigma(\beta t+x_1)}a^{-{d-2\over 4}}(x)l(x)\chi(x)$, $(t,x)\in\Sigma_{-}$.  Since $v_-(t,x)=0$ for $t\in(0,T)$, $x\in \{x\in\partial M:\ \pd_\nu\phi(x)\geq-\epsilon/3\}$  we have
$v_-\in L_{-\sigma,\pd_\nu\phi^{-1},-}$. Fix also $v=-(\Box_{a,g}+q)(e^{\sigma(\beta t+x_1)}a^{-{d-2\over 4}}l)$ and $v_0(x)=-a^{-{d-2\over 4}}l(x)e^{\sigma x_1}$, $(t,x)\in Q$. From Lemma \ref{l3}, we deduce that there exists $w\in H_{\Box_{a,g}}((0,T)\times M)$ such that
\[
\left\{
\begin{array}{ll}
(a^{-1}\partial_t^2-\Delta_g+q) w=v=-(\Box_{a,g}+q)(e^{\sigma(\beta t+x_1)}a^{-{d-2\over 4}}l)&  \mbox{in}\; (0,T)\times M,
\\
w(0,x)=v_0(x)=-l(x)a^{-{d-2\over 4}}(x)e^{\sigma x_1}, &  x\in M,
\\
w(t,x)=v_-(t,x)=-e^{\sigma(\beta t+x_1)}a^{-{d-2\over 4}}(x)l(x)\chi(x),& (t,x)\in\Sigma_{-}.
\end{array}
\right.\]
Then, for $z=e^{-\sigma(\beta t+x_1)} w$ condition \eqref{w1} will be fulfilled. Moreover,  condition 3) of Lemma \ref{l3} and \eqref{t3d} imply
\[\begin{aligned}&\norm{z}_{L^2((0,T)\times M)}=\norm{w}_{-\sigma}\leq C\left(\sigma^{-1}\norm{v}_{-\sigma}+\sigma^{-\frac{1}{2}}\norm{v_-}_{-\sigma,\pd_\nu\phi^{-1},-}+\sigma^{-\frac{1}{2}}\norm{v_0}_{-\sigma,0}\right)\\
\ &\leq C\left(\sigma^{-1}(1+\norm{q}_{L^2((0,T)\times M)})+\sigma^{-\frac{1}{2}}\norm{a^{-{d-2\over 4}}\chi l\pd_\nu\phi^{-1/2}}_{L^2(\Sigma_{-})}+\sigma^{-\frac{1}{2}}\norm{a^{-{d-2\over 4}}l}_{L^2(M)}\right)\\
\ &\leq C\sigma^{-{1\over2}}\end{aligned}\]
with $C$ depending only on $\beta$, $U$, $M$, $T$ and $\norm{q}_{L^\infty((0,T)\times M)}$. Therefore,  estimate \eqref{CGO1b} holds.  Using the fact that $e^{\sigma(\beta t+x_1)} z=w\in H_{\Box_{a,g}}((0,T)\times M)$, we deduce that  $u$ defined by \eqref{CGO1a} is lying in $H_{\Box_{a,g}}((0,T)\times M)$ and is a solution of \eqref{(5.1)}. This completes the proof of Theorem \ref{t3}.
\end{proof}

\section{Proof of Theorem \ref{thm1}}
From now on we set $q=q_2-q_1$ on $(0,T)\times M$ and  we assume  that $q=0$ on $\R^2\times D\setminus ((0,T)\times M)$. Without loss of generality we assume that   we have $\partial M_{-,\epsilon,+}\subset V'$ with $\epsilon>0$ introduced in the beginning of the previous section and we fix   $\sigma >\max(\sigma_1,\sigma_0) $. According to Proposition \ref{p2}, we can introduce
\[u_1(t,x)=e^{-\sigma(\beta t+x_1)}a^{-{d-2\over 4}}(x)\left(k(t,x)+w(t,x) \right),\ (t,x) \in (0,T)\times M,\]
where $u_1\in H^1((0,T)\times M)$ satisfies $a^{-1}\partial_t^2u_1-\Delta_gu_1+q_1u_1=0$, $k$ given by \eqref{GO2} and  $w$ satisfies \eqref{p2a}. Moreover, in view of Theorem \ref{t3}, we consider $u_2\in H_{\Box_{a,g}}((0,T)\times M)$ a solution of \eqref{(5.1)} with $q=q_2$ of the form 
\[u_2(t,x)=e^{\sigma(\beta t+x_1)}\left(a^{-{d-2\over 4}}l(x)+z(t,x) \right),\quad (t,x)\in (0,T)\times M\]
with $l$ given by \eqref{GGO},  $z$ satisfying \eqref{CGO1b}, such that supp$(u_{2|(0,T)\times\pd M}) \subset U$ and $u_{2|t=0}=0$.
In view of Proposition \ref{p6}, there exists a unique solution $w_1\in H_{\Box_{a,g}} ((0,T)\times M)$ of
 \bel{eq3}
\left\{
\begin{array}{ll}
 a^{-1}\partial_t^2w_1-\Delta_gw_1 +q_1w_1=0 &\mbox{in}\ (0,T)\times M,
\\

\tau_{0}w_1=\tau_{0}u_2. &

\end{array}
\right.
\ee
Then, $u=w_1-u_2$ solves
  \bel{eq4}
\left\{\begin{array}{ll}
 a^{-1}\partial_t^2u-\Delta_gu +q_1u=(q_2-q_1)u_2 &\mbox{in}\ (0,T)\times M,
\\
u(0,x)=\partial_tu(0,x)=0 & \mathrm{on}\  M,\\

u=0 &\mbox{on}\ (0,T)\times\pd M
\end{array}\right.
\ee
and since $(q_2-q_1)u_2\in L^2((0,T)\times M)$, in view of     \cite[Theorem 2.1]{LLT} and the formula \eqref{red}, we deduce that $u\in \mathcal H$ and $\pd_\nu u\in L^2((0,T)\times \pd M)$.  We use again the notation $\overline{M}=[0,T]\times M$ and $\overline{g}=dt^2+g$. Since $u,\ u_1\in H^1((0,T)\times M)$ and $a^{-1}\pd_t^2u-\Delta_gu\in L^2((0,T)\times M)$, one can check that $(a^{-1}\pd_tu,-\nabla_gu)\in H_{div}=\{F\in L^2(\overline{M};T\overline{M}):\ div_{\overline{g}}F\in L^2(\overline{M})\}$. Therefore, by density one can check that
$$\begin{array}{l}\int_{\overline{M}}(a^{-1}\pd_t^2u-\Delta_gu)u_1dV_g(x)dt\\
\ \\
=-\int_{\overline{M}}\left\langle (a^{-1}\pd_tu,-\nabla_gu),\nabla_{\overline{g}}u_1\right\rangle_{\overline{g}}dV_g(x)dt+\left\langle \left\langle (a^{-1}\pd_tu,-\nabla_gu),\overline{\nu}\right\rangle_{\overline{g}},u_1\right\rangle_{H^{-{1\over2}}(\pd\overline{M}),H^{1\over2}(\pd\overline{M})}\end{array}$$
with $\overline{\nu}$ the outward unit normal vector to $\overline{M}$. From this formula we deduce that
$$\begin{array}{l}\int_0^T\int_{M}(a^{-1}\pd_t^2u-\Delta_gu)u_1dV_g(x)dt\\
\ \\
=-\int_0^T\int_{M}(a^{-1}\pd_tu\pd_tu_1-\left\langle \nabla_{g}u,\nabla_{g}u_1\right\rangle_{g})dV_g(x)dt+\left\langle \left\langle (a^{-1}\pd_tu,-\nabla_gu),\overline{\nu}\right\rangle_{\overline{g}},u_1\right\rangle_{H^{-{1\over2}}(\pd\overline{M}),H^{1\over2}(\pd\overline{M})}.\end{array}$$
In the same way, we find 
$$\begin{array}{l}\int_0^T\int_{M}(a^{-1}\pd_t^2u_1-\Delta_gu_1)udV_g(x)dt\\
\ \\
=-\int_0^T\int_{M}(a^{-1}\pd_tu\pd_tu_1-\left\langle \nabla_{g}u,\nabla_{g}u_1\right\rangle_{g})dV_g(x)dt+\left\langle \left\langle (a^{-1}\pd_tu_1,-\nabla_gu_1),\overline{\nu}\right\rangle_{\overline{g}},u\right\rangle_{H^{-{1\over2}}(\pd\overline{M}),H^{1\over2}(\pd\overline{M})}.\end{array}$$
Combining these two formulas, we get
$$\begin{array}{l}\int_0^T\int_{M}(q_2-q_1)u_2dV_g(x)dt\\
\ \\
=\int_0^T\int_{M}(a^{-1}\pd_t^2u-\Delta_gu+q_1u)u_1dV_g(x)dt-\int_0^T\int_{M}(a^{-1}\pd_t^2u_1-\Delta_gu_1+q_1u_1)udV_g(x)dt\\
\ \\
=\left\langle \left\langle (a^{-1}\pd_tu,-\nabla_gu),\overline{\nu}\right\rangle_{\overline{g}},u_1\right\rangle_{H^{-{1\over2}}(\pd\overline{M}),H^{1\over2}(\pd\overline{M})}-
\left\langle \left\langle (a^{-1}\pd_tu_1,-\nabla_gu_1),\overline{\nu}\right\rangle_{\overline{g}},u\right\rangle_{H^{-{1\over2}}(\pd\overline{M}),H^{1\over2}(\pd\overline{M})}\end{array}$$
On the other hand, condition \eqref{thm1a} implies that $u_{|t=T}=\partial_\nu u_{|V}=0$ and  we obtain
\begin{equation}\label{t3a} \int_0^T\int_Mqu_2u_1dV_g(x)dt=-\int_{(0,T)\times\pd M\setminus V}\partial_\nu uu_1d\sigma_g(x)dt+\int_M a^{-1}\partial_tu(T,x)u_1(T,x)dV_g(x).\end{equation}
Applying  the Cauchy-Schwarz inequality to the first expression on the right hand side of this formula, we get
\[\begin{aligned}\abs{\int_{(0,T)\times\pd M\setminus G}\partial_\nu uu_1d\sigma_g(x)dt}&\leq\int_{{\Sigma}_{+,\epsilon,+}}\abs{\partial_\nu ue^{-\sigma(\beta t+x_1)}(k(t,x)+w)}d\sigma_g(x)dt  \\
 \ &\leq C\left(\int_{{\Sigma}_{+,\epsilon,+}}\abs{e^{-\sigma(\beta t+_1x)}\partial_\nu u}^2d\sigma_g(x)dt\right)^{\frac{1}{2}}\end{aligned}\]
for some $C$ independent of $\sigma$. Here we have used both \eqref{p2a}, the fact that $|k|$ is independent of $\sigma$ and the fact that $((0,T)\times\pd M\setminus V)\subset {\Sigma}_{+,\epsilon,+}$. In the same way, we have
\[\begin{aligned}\abs{\int_M a^{-1}\partial_tu(T,x)u_1(T,x)dV_g(x)}&\leq\norm{a^{-1}}_{L^\infty(M)}\int_{M}\abs{\partial_t u(T,x)e^{-\sigma(\beta t+x_1)}\left(k(T,x)+w(T,x)\right)}dV_g(x) \\
 \ &\leq C\left(\int_{M}\abs{e^{-\sigma(\beta t+x_1)}\partial_t u(T,x)}^2dV_g(x)\right)^{\frac{1}{2}}.\end{aligned}\]
Combining these estimates with the Carleman estimate \eqref{c1a} and the facts that $u_{|t=T}=\partial_\nu u_{|\Sigma_{-}}=0$ and ${\partial M}_{+,\epsilon,+}\subset {\partial M}_{+}$, we find
\begin{eqnarray}&&\abs{\int_0^T\int_M(q_2-q_1)u_2u_1dV_g(x)dt}^2\cr
&&\leq 2C\left(\int_{{\Sigma}_{+,\epsilon,+}}\abs{e^{-\sigma(\beta t+x_1)}\partial_\nu u}^2d\sigma_g(x)dt+\int_M \abs{e^{-\sigma(\beta t+x_1)}\partial_tu(T,x)}^2dV_g(x)\right)\cr
&&\leq 2\epsilon^{-1}C\left(\int_{{\Sigma}_{+}}\abs{e^{-\sigma(\beta t+x_1)}\partial_\nu u}^2\pd_\nu\phi(x) d\sigma_g(x)dt+\int_M \abs{e^{-\sigma(\beta T+x_1)}\partial_tu(T,x)}^2dV_g(x)\right)\cr
&&\leq {\epsilon^{-1}C\over\sigma}\left(\int_0^T\int_M\abs{ e^{-\sigma(\beta t+x_1)}(a^{-1}\partial_t^2-\Delta_x +q_1)u}^2dV_g(x)dt\right)\cr
&&\leq {\epsilon^{-1}C\over\sigma}\left(\int_0^T\int_M\abs{ e^{-\sigma(\beta t+x_1)}qu_2}^2dV_g(x)dt\right)={\epsilon^{-1}C\over\sigma}\left(\int_0^T\int_M\abs{ q}^2(1+\abs{z})^2dV_g(x)dt\right).\end{eqnarray}
Here $C>0$ stands for some generic constant independent of $\sigma$.
It follows that
\begin{equation}\label{t3cc}\lim_{\sigma\to+\infty}\int_Qqu_2u_1dtdx=0.\end{equation}
On the other hand, using the fact that $q=0$ on $\R^2\times M_0\setminus( (0,T)\times M)$ we have
\begin{align*}
&\int_{0}^T\int_Mqu_1u_2d V_g(x)dt
\\&\quad=\int_{\R}\int_\R \int_{M_0}q(t,x_1,x')k(t,x_1,x')a^{-{d-2\over 2}}l(x_1,x')dV_{g'}(x')dx_1dt+ \int_0^T\int_MZ(t,x) dV_g(x)dt
\end{align*}
with $ Z(t,x)=q(t,x)(a^{-{d-2\over 4}}z(t,x)k(t,x)+a^{-{d-2\over 2}}w(t,x)l(x)+a^{-{d-2\over 4}}z(t,x)w(t,x))$. Then, in view of \eqref{p2a}  and \eqref{CGO1b}, an application of the Cauchy-Schwarz inequality yields
\[\abs{\int_0^T\int_MZ(t,x)  dV_g(x)dt}\leq C\sigma^{-\frac{1}{2}}\]
with $C$ independent of $\sigma$. Combining this with \eqref{t3cc}, we deduce that 
\[\int_{\R}\int_\R \int_{M_0}a^{-{d-2\over 2}}(x_1,x')q(t,x_1,x')k(t,x_1,x')l(x_1,x')dV_{g'}(x')dx_1dt=0.\]
Fixing $y\in\pd D$, $\tau_+(y,\theta)$ the time of existence in $D$ of the maximal geodesic $\gamma_{y,\theta}$ satisfying $\gamma_{y,\theta}(0)=y$ and $\gamma_{y,\theta}'(0)=\theta$, we obtain in polar normal coordinate
\[\int_{\R}\int_\R\int_{S_yM_1}\int_{0}^{\tau_+(y,\theta)}a^{-{d-2\over 2}}(x_1,r,\theta)q(t,x_1,r,\theta)k(t,x_1,r,\theta)l(x_1,r,\theta)dV_{\tilde{g}}(r,\theta)dx_1dt=0\]
with $dV_{\tilde{g}}$ the Riemanian volume form in polar normal coordinate given by $b (r,\theta)^{1/2}drd\theta$. It follows from \eqref{GO2} and \eqref{GGO}, that   for any $h\in H^2(S_yM_1)$, $\mu\in\R$ and $\beta\in[1/2,1]$, we have
\[\int_{S_yM_1}\left(\int_{\R}\int_\R\int_{0}^{\tau_+(y,\theta)}a^{-{d-2\over 2}}(x_1,r,\theta)q(t,x_1,r,\theta)e^{-i{\mu (t+\beta x_1)}}drdx_1dt\right)h(\theta)d\theta=0\]
which implies that
\bel{t1d}\int_{\R}\int_\R\int_{0}^{\tau_+(y,\theta)}a^{-{d-2\over 2}}(x_1,r,\theta)q(t,x_1,r,\theta)e^{-i{\mu (t+\beta x_1)}}drdx_1dt=0,\quad \theta\in S_yM_1.\ee
We recall that the  geodesic ray transform $I$ on $\pd _+SD:=\{(x,\xi)\in SD:\ x\in\pd D,\ \left\langle \xi,\nu(x)\right\rangle_{g'}<0\}$ is defined by
\[I f(t,x_1,x',\theta)=\int_0^{\tau_+(x,\theta)}f(\gamma_{x',\theta}(t))dt,\quad (x',\theta)\in \pd _+SD,\ \ f\in L^1(\R\times\R\times D).\]
Fixing $\mathcal F_{t,x_1}a^{-{d-2\over 2}}q$ the partial Fourier transform of $a^{-{d-2\over 2}}q$ with respect to $t\in\R$ and $x_1\in\R$ given by 
\[\mathcal F_{t,x_1}a^{-{d-2\over 2}}q(\xi_1,\xi_2,x')=(2\pi)^{-1}\int_{\R}\int_\R a^{-{d-2\over 2}}(x_1,x')q(t,x_1,x')e^{-i{(\xi_1t+\xi_2x_1)}}dx_1dt,\quad \xi=(\xi_1,\xi_2)\in \R^2\]
we obtain from \eqref{t1d} that for any $\mu\in \R$ and all $\beta\in[{1\over2},1]$
\bel{end}\mathcal F_{t,x_1} [I(a^{-{d-2\over 2}}q)(\cdot,\cdot,y,\theta)](\mu,b\mu)=\int_{\R}\int_\R\int_{0}^{\tau_+(y,\theta)}a^{-{d-2\over 2}}(x_1,r,\theta)q(t,x_1,r,\theta)e^{-i{\mu (t+\beta x_1)}}drdx_1dt=0\ee
holds true.
Since for all $\theta\in S_y(D)$, such that $(y,\theta)\in\pd _+SD$, the function $Q:(t,x_1)\mapsto I(a^{-{d-2\over 2}}q)(t,x_1,y,\theta)\in L^\infty(\R^2)$ is supported on $[0,T]\times [-R,R]$, its Fourier transform is analytic 
and \eqref{end} implies that $Q=0$. Therefore, for almost every $(t,x_1)\in\R^2$ we have $I(a^{-{d-2\over 2}}q)(t,x_1,y,\theta)=0$.
Using the injectivity of the geodesic ray transform (e.g. \cite[Proposition 7.2]{DKSU}) , by varying $y\in \pd D$, we deduce that  $q=0$. This completes the proof of Theorem \ref{thm1}.

\vspace{0.5cm}
\noindent{\bf Acknowledgements.}
LO was partly supported by EPSRC grant EP/L026473/1.

\noindent {\sc Yavar Kian}, Aix Marseille Univ, Univ Toulon, CNRS, CPT, Marseille, France.\\
E-mail: {\tt yavar.kian@univ-amu.fr}. \vspace*{.5cm} \\
\noindent {\sc Lauri Oksanen}, Department of Mathematics, University College London, Gower Street, London, WC1E 6BT, UK.\\
E-mail: {\tt l.oksanen@ucl.ac.uk}. 

\end{document}